\documentclass[a4paper, 10pt]{amsart}
\normalsize

\usepackage{amssymb}
\usepackage{bbm}
\usepackage[T5,T1]{fontenc}

\newlength{\aufzleft}

\newenvironment{equi}{\begin{list}{}{\setlength{\listparindent}{0pt}\setlength{\itemsep}{\topsep}\setlength{\labelwidth}{4.1ex}\setlength{\aufzleft}{\labelsep}\addtolength{\aufzleft}{\labelwidth}\setlength{\leftmargin}{\aufzleft}}}{\end{list}}

\newtheoremstyle{bracket}{1ex}{2ex}{\rm}{}{\bfseries}{}{0.8em}{\thmnumber{(#2)}}
\newtheoremstyle{example}{1ex}{2ex}{\rm}{}{\bfseries}{}{0.8em}{\thmnumber{(#2)}\thmname{ #1}}
\newtheoremstyle{thm}{1ex}{2ex}{\itshape}{}{\bfseries}{}{0.9em}{\thmnumber{(#2)}\thmname{ #1}\thmnote{ (#3)}}

\theoremstyle{bracket}
\newtheorem{no}{}[section]

\theoremstyle{example}
\newtheorem{example}[no]{Example}
\newtheorem{examples}[no]{Examples}

\theoremstyle{thm}
\newtheorem{lemma}[no]{Lemma}
\newtheorem{proposition}[no]{Proposition}
\newtheorem{theorem}[no]{Theorem}
\newtheorem{corollary}[no]{Corollary}

\DeclareMathOperator{\adj}{adj}
\DeclareMathOperator{\card}{Card}
\DeclareMathOperator{\cl}{cl}
\DeclareMathOperator{\cone}{cone}
\DeclareMathOperator{\conv}{conv}
\DeclareMathOperator{\exte}{ex}
\DeclareMathOperator{\face}{face}
\DeclareMathOperator{\fr}{fr}
\DeclareMathOperator{\inn}{in}
\DeclareMathOperator{\ke}{Ker}
\DeclareMathOperator{\pface}{pface}

\newcommand{\dd}{\mathfrak{D}}
\newcommand{\ff}{\mathfrak{F}}
\newcommand{\eps}{\varepsilon}
\newcommand{\thet}{\vartheta}
\newcommand{\sig}{\Sigma}
\newcommand{\Tau}{{\rm T}}
\newcommand{\Nn}{\mathbbm{N}}
\newcommand{\R}{\mathbbm{R}}
\newcommand{\C}{\mathbbm{C}}
\newcommand{\Q}{\mathbbm{Q}}
\newcommand{\Z}{\mathbbm{Z}}
\newcommand{\fff}{\overline{\mathfrak{F}}}
\newcommand{\hap}{\widehat{p}}
\newcommand{\omegaq}{\overline{\Omega}}
\newcommand{\psiquer}{\overline{\Psi}}
\newcommand{\sighut}{\widehat{\Sigma}}
\newcommand{\sigquer}{\overline{\Sigma}}
\newcommand{\tauhut}{\widehat{\Tau}}
\newcommand{\tauquer}{\overline{\Tau}}
\newcommand{\V}{\mathfrak{V}}
\newcommand{\dfgl}{\mathrel{\mathop:}=}
\newcommand{\fleq}{\preccurlyeq}
\newcommand{\flneq}{\prec}
\newcommand{\norm}{\parallel\!\cdot\!\parallel}
\newcommand{\nrm}[1]{||#1||}
\newcommand{\streckel}{[\hskip-0.33ex[}
\newcommand{\strecker}{]\hskip-0.33ex]}

\begin{document}

\title{Completions of Fans}
\author{Fred Rohrer}
\address{Institute of Mathematics\\Vietnam Academy of Science and Technology\\{\fontencoding{T5}\selectfont 18 Ho\`ang Qu\'\ocircumflex{}c Vi\d\ecircumflex{}t\\10307 H\`a~N\d\ocircumflex{}i\\Vi\d\ecircumflex{}t Nam}}
\email{fredrohrer0@gmail.com}
\thanks{The author was supported by the Swiss National Science Foundation.}


\begin{abstract}
In a finite-dimensional real vector space furnished with a rational structure with respect to a subfield of the field of real numbers, every (simplicial) rational semifan is contained in a complete (simplicial) rational semifan. In this paper this result is proved constructively on use of techniques from polyhedral geometry.
\end{abstract}

\maketitle


\section*{Introduction}

Let $V$ be a finite-dimensional $\R$-vector space. A fan in $V$ is defined as a finite set $\sig$ of sharp polycones in $V$ (that is, intersections of finitely many closed linear halfspaces in $V$ that do not contain a line), closed under taking faces (that is, intersections of a polycone $\sigma$ with a linear hyperplane $H$ such that $\sigma$ lies on one side of $H$) and such that the intersection of two polycones in $\sig$ is a face of both. A fan $\sig$ in $V$ is called complete if its support $\bigcup\sig$ equals $V$. A natural question is whether every fan has a completion, that is, is contained in a complete fan.

This question is also of fundamental importance in toric algebraic geometry. Namely, a fan $\sig$ that is rational with respect to some $\Q$-structure on $V$ corresponds to a toric variety $X_{\sig}$ over $\C$, and this variety is complete (in the sense of algebraic geometry) if and only if the fan $\sig$ is complete or empty. Hence, a completion $\sighut$ of $\sig$ yields an open immersion $X_{\sig}\rightarrowtail X_{\sighut}$ from $X_{\sig}$ into a complete toric variety. More generally, if $R$ is an arbitrary commutative ring then the fan $\sig$ gives rise to a toric $R$-scheme $X_{\sig}(R)$, and a completion $\sighut$ of $\sig$ yields an open immersion $X_{\sig}(R)\rightarrowtail X_{\sighut}(R)$ from $X_{\sig}(R)$ into a proper toric $R$-scheme. In this sense, completions of fans correspond to ``compactifications'' in toric geometry.

Interestingly, it was by this detour into algebraic geometry that the first proof of existence of completions of fans was given. More precisely, in 1974 Sumihiro published his Equivariant Compactification Theorem (\cite{sumihiro}). But it seems to be mentioned for the first time only in 1988 in Oda's comprehensive book on toric varieties (\cite[p.~17]{oda}) that on use of Sumihiro's theorem on can get a proof of the existence of completions of fans that are rational with respect to some $\Q$-structure. However, this proof is not constructive, and it leaves us wondering if such a simple statement about polyhedral geometry can be proven without relying on the heavy machinery from algebraic geometry. Concerning this point, Oda stated in \textit{loc.cit.} that ``no systematic way of constructing [a completion of a given fan] seems to be known in general''. Around the same time, Ewald stated in \cite{ew0} the existence of completions of fans without proof, but in a way that suggests a direct and constructive proof. In 1996 he indeed sketched such a proof in his textbook on combinatorial convexity (\cite[III.2.8]{ew}), and finally ten years later Ewald and Ishida published in \cite{ei} a refined version of this construction. (In \textit{loc.cit.} a second proof is given that imitates combinatorially the algebro-geometric proof and thus is not what we are looking for.) Although the construction of Ewald and Ishida relies on very natural ideas it is technically quite involved. Moreover, the fact that there do not exist canonical completions (as is seen immediately with some easy examples) suggests that every construction of completions will be complicated.
 
In this paper, we take up the ideas of Ewald and Ishida and modify them in order to arrive at a simplified construction of completions of fans. Besides greater simplicity (for example in avoiding the polyhedral complexes obtained by projecting a fan and hence avoiding the construction of the map $\phi_a$ (\cite[Lemma 1.3]{ei})) our construction has several further advantages. First, some pathological fans (for example fans in a space of dimension greater than $2$ having a $1$-dimensional maximal cone) do not have to be treated separately by means of ad hoc constructions. Second, our arguments rely to a large extent on the canonical structure of topological vector space on the ambient space $V$ and avoid the use of some norm on $V$ -- this improves clarity and might moreover be useful for studying fans in more general settings. Third, our approach is very conceptual and thus might be helpful in studying completions of fans preserving additional properties.

\smallskip

We now give an overview of the contents of this article. In Section \ref{sec1} we review the terminology and collect basic properties of polycones and fans. Moreover, we introduce and study a notion of direct sum of polycones that yields a reasonable decomposition into indecomposable polycones. Section \ref{sec3} treats some topological properties of fans, the main result being a combinatorial description of the topological frontier of the support of a fan. In Section \ref{sec5} we introduce the notions of relatively simplicial, separable and tightly separable extensions of fans, which are crucial for our construction of completions. We treat three particular constructions of extensions of fans in Section \ref{sec6}, and we apply them in Section \ref{sec9} in order to prove existence of completions by an inductive and recursive construction.

\smallskip

The content of this article is part of the author's doctoral dissertation \cite{diss} which is available at {\tt www.dissertationen.uzh.ch}.


\section{Polycones and fans}\label{sec1}

\textit{Throughout this article let $V$ be an $\R$-vector space of finite dimension $n$, and let $K\subseteq\R$ be a subfield. If not specified otherwise a morphism is a morphism of $\R$-vector spaces, a subspace is a sub-$\R$-vector space, an affine subspace is an affine sub-$\R$-space, and a section is an $\R$-linear section.}\medskip

We fix some terminology and review the notion of rational structures on finite-dimensional $\R$-vector spaces from \cite[II.8]{a}.\medskip

\begin{no}
We denote by $V^*$ and $V^{**}$ the dual and bidual of $V$, and we identify $V$ and $V^{**}$ by means of the canonical isomorphism $c_V:V\overset{\cong}\longrightarrow V^{**}$. For $A\subseteq V$ we denote by $\langle A\rangle$ the subspace of $V$ generated by $A$, and by $\dim(A)$ the dimension of the affine sub-$\R$-space of $V$ generated by $A$. Moreover, we write $A^{\perp,V}$ and $A^{\vee,V}$ (or just $A^{\perp}$ and $A^{\vee}$) for the orthogonal $\{u\in V^*\mid u(A)\subseteq 0\}$ and the dual $\{u\in V^*\mid u(A)\subseteq\R_{\geq 0}\}$ of $A$; in case $A=\{x\}$ we write $x^{\perp}$ and $x^{\vee}$.

Furthermore, we furnish $V$ with its canonical topology. For $A\subseteq V$ we denote by $\inn_V(A)$, $\cl_V(A)$ and $\fr_V(A)$ (or just by $\inn(A)$, $\cl(A)$ and $\fr(A)$) the interior, closure and frontier of $A$. For $x\in V$ we denote by $\V_V(x)$ (or just by $\V(x)$) the filter of neighbourhoods of $x$.

By means of scalar restriction we consider $V$ as a $K$-vector space. A \textit{$K$-structure on $V$} is a sub-$K$-vector space $W\subseteq V$ such that the canonical morphism $\R\otimes_KW\rightarrow V$ with $a\otimes x\mapsto ax$ is an isomorphism, or -- equivalently -- that $\langle W\rangle=V$ and $\dim_K(W)=n$.
\end{no}

\textit{From now on let $W$ be a $K$-structure on $V$.}\medskip

\begin{no}\label{compl}
A subspace of $V$ is called \textit{$W$-rational} if it has a basis contained in $W$, and an affine subspace of $V$ is called \textit{$W$-rational} if it is the translation by an $x\in W$ of a $W$-rational subspace of $V$. A $W$-rational (affine) hyperplane in $V$ is called an \textit{(affine) $W$-hyperplane in $V$.} A closed (affine) halfspace defined by an (affine) $W$-hyperplane in $V$ is called an \textit{(affine) $W$-halfspace in $V$.} If $V'\subseteq V$ is a $W$-rational subspace, then $W'\dfgl W\cap V'$ is a $K$-structure on $V'$ called \textit{the induced $K$-structure,} and the $K$-module $W/W'$ is canonically isomorphic to and identified with a $K$-structure on $V/V'$, denoted by abuse of language by $W/V'$. Every $W$-rational subspace $V'\subseteq V$ has a $W$-rational complement.

If $V'$ is a finite-dimensional $\R$-vector space and $W'$ is a $K$-structure on $V'$, then a morphism $V\rightarrow V'$ is called \textit{$(W,W')$-rational} if it induces by restriction and coastriction a morphism $W\rightarrow W'$ of $K$-vector spaces. The set of $(W,K)$-rational linear forms on $V$ is a $K$-structure on $V^*$, canonically isomorphic to and identified with the dual $W^*$ of $W$. Applying this to $V^*$ we get the $K$-structure $W^{**}$ on $V^{**}$, and $c_V$ being $(W,W^{**})$-rational we identify $W$ and $W^{**}$.
\end{no}

\begin{no}\label{hilb}
A Hilbert norm on $V$ is called \textit{$W$-rational} if its corresponding $\R$-bilinear form on $V$ induces a $K$-bilinear form on $W$. Clearly, there exists a $W$-rational Hilbert norm on $V$. If $V'\subseteq V$ is a $W$-rational subspace, then a $W$-rational Hilbert norm $\norm$ on $V$ induces a canonical $W/V'$-rational Hilbert norm on $V/V'$ and a $(W/V',W)$-rational section of the canonical projection $V\twoheadrightarrow V/V'$ with image the orthogonal complement of $V'$ in $V$ with respect to $\norm$. (To see this, one should note that the orthogonal complement of $V'$ is $W$-rational since it equals $\bigcap_{y\in W\cap V'}\ke(f(\cdot,y))$, where $f$ denotes the $\R$-bilinear form corresponding to $\norm$.)
\end{no}

Now we fix further terminology and collect basic properties of polycones and fans. These are proven in some form in most of the textbooks on polyhedral geometry, and in precisely this form in \cite{diss}.\medskip

\begin{no}
A \textit{$W$-polycone (in $V$)} is the intersection of finitely many $W$-halfspaces in $V$, and it is called \textit{sharp} if it does not contain a line. For $A\subseteq V$ we denote by $\cone(A)$ the conic hull of $A$, that is, the set of $\R$-linear combinations of $A$ with coefficients in $\R_{\geq 0}$. A subset $\sigma\subseteq V$ is a $W$-polycone if and only if there is a finite subset $A\subseteq W$ with $\sigma=\cone(A)$.

Now, let $\sigma$ be a $W$-polycone. Then, $\langle\sigma\rangle=\sigma-\sigma$, and $s(\sigma)\dfgl\sigma\cap(-\sigma)$ is the greatest subspace of $V$ contained in $\sigma$; both these subspaces are $W$-rational. Furthermore, $\sigma$ is called \textit{full (in $V$)} if $\dim(\sigma)=n$. A \textit{(proper) face of $\sigma$} is a (proper) subset $\tau\subseteq\sigma$ such that there exists $u\in\sigma^{\vee}$ (or equivalently $u\in\sigma^{\vee}\cap W^*$) with $\tau=\sigma\cap u^{\perp}$. Faces of $\sigma$ are again $W$-polycones, and we denote by $\face(\sigma)$ and $\pface(\sigma)$ the finite sets of faces and of proper faces of $\sigma$. The relation ``$\tau$ is a face of $\sigma$'' is an ordering on the set of $W$-polycones and is denoted by $\tau\fleq\sigma$. If $k\in\Z$ and no confusion can arise then we set $\sigma_k\dfgl\{\tau\in\face(\sigma)\mid\dim(\tau)=k\}$. The $W$-polycone $\sigma$ is called \textit{simplicial} if there is a free subset $A\subseteq V$ with $\sigma=\cone(A)$, and then $\face(\sigma)=\{\cone(B)\mid B\subseteq A\}$. Simplicial $W$-polycones are sharp, and faces of simplicial $W$-polycones are again simplicial.

For $A\subseteq V$ we denote by $\conv(A)$ the convex hull of $A$. For $x,y\in V$ we set $\streckel x,y\strecker\dfgl\conv(\{x,y\})$, $\,\strecker x,y\strecker\dfgl\streckel x,y\strecker\setminus\{x\}$ and $\,\strecker x,y\streckel\,\dfgl\streckel x,y\strecker\setminus\{x,y\}$.
\end{no}

\begin{no}\label{topprop}	
a) We will make use of the following topological properties of polycones: A $W$-polycone $\sigma$ is full if and only if $\inn(\sigma)\neq\emptyset$, and then $\cl(\inn(\sigma))=\sigma$ and $\fr(\inn(\sigma))=\fr(\sigma)=\bigcup\pface(\sigma)$. Moreover, $\inn(\sigma)$ is convex, and if an open subset of $V$ meets a $W$-polycone $\sigma$ then it meets $\inn_{\langle\sigma\rangle}(\sigma)$.

b) We will make use of the following combinatorial properties of polycones: A $W$-polycone $\sigma$ is sharp if and only if $0\in\face(\sigma)$, and then $\sigma=\sum\sigma_1$. Moreover, if $\sigma$ is a full $W$-polycone, then every face of $\sigma$ is the intersection of a family in $\sigma_{n-1}$.
\end{no}

\begin{no}\label{sep}
Let $A,B\subseteq V$ be subsets. If there is an affine $W$-hyperplane $H$ in $V$ that separates $A$ and $B$ (strictly) then $A$ and $B$ are called \textit{(strictly) $W$-separable,} and if moreover $H$ \textit{separates them in their intersection,} that is, if $A\cap H=A\cap B=B\cap H$, then they are called \textit{$W$-separable in their intersection.}

Two $W$-polycones $\sigma$ and $\tau$ such that $\sigma\cap\tau$ is not full are $W$-separable in their intersection if and only if $\sigma\cap\tau\in\face(\sigma)\cap\face(\tau)$.
\end{no}

\begin{no}
A \textit{$W$-semifan (in $V$)} is a finite set $\sig$ of $W$-polycones such that $\sigma\cap\tau\in\face(\sigma)\subseteq\sig$ for all $\sigma,\tau\in\sig$, and a \textit{$W$-fan (in $V$)} is a $W$-semifan $\sig$ such that all its elements are sharp. A $W$-semifan is considered as an ordered set by means of the ordering induced by $\sigma\fleq\tau$.

Now, let $\sig$ be a $W$-semifan. We set $s(\sig)\dfgl\bigcap\sig$; if $\sig\neq\emptyset$ then this is the smallest element of $\sig$. The set of maximal elements of $\sig$ is denoted by $\sig_{\max}$. For $k\in\Z$ we set $\sig_k\dfgl\{\sigma\in\sig\mid\dim(\sigma)=k\}$, and we define $\sig_{\geq k}$ and $\sig_{<k}$ analogously. If $\sig_{\max}=\sig_n\neq\emptyset$ then $\sig$ is called \textit{equifulldimensional,} and if every $\sigma\in\sig$ is simplicial then $\sig$ is called \textit{simplicial.} For $\sigma\in\sig$ we set $\sig_{\sigma}\dfgl\{\tau\in\sig\mid\sigma\fleq\tau\}$. Furthermore, we set $\dd(\sig)\dfgl\sig_{\max}\setminus\sig_n$ and $\ff(\sig)\dfgl\{\sigma\in\sig_{n-1}\mid\exists!\tau\in\sig_n:\sigma\fleq\tau\}$, and moreover $\fff(\sig)\dfgl\bigcup_{\sigma\in\ff(\sig)}\face(\sigma)$. We also set $|\sig|\dfgl\bigcup\sig$, and we call $\sig$ \textit{full (in $V$)} if $\langle|\sig|\rangle=V$ and \textit{complete (in $V$)} if $|\sig|=V$.

If $\sig$ and $\sig'$ are $W$-semifans with $\sig\subseteq\sig'$, then $\sig$ is called a \textit{$W$-subsemifan of $\sig'$} and $\sig'$ is called a \textit{$W$-extension of $\sig$;} if in addition $\sig'$ is complete then it is called a \textit{$W$-completion of $\sig$.} In case $\sig$ and $\sig'$ are $W$-fans we speak of $W$-subfans.
\end{no}

\begin{no}\label{proj}
Let $\sig$ be a $W$-semifan, and let $\sigma\in\sig$. We set $V_{\sigma}\dfgl V/\langle\sigma\rangle$ and $W_{\sigma}\dfgl W/\langle\sigma\rangle$, and denote by $p_{\sigma}:V\twoheadrightarrow V_{\sigma}$ the canonical epimorphism. Let $\sig/\sigma\dfgl\{p_{\sigma}(\tau)\mid\tau\in\sig_{\sigma}\}$. This is a $W_{\sigma}$-fan, and for $\tau\in\sig_{\sigma}$ we write by abuse of language $\tau/\sigma\dfgl p_{\sigma}(\tau)$. The map $\hap_{\sigma}:\sig_{\sigma}\rightarrow\sig/\sigma,\;\tau\mapsto\tau/\sigma$ is an isomorphism of ordered sets inducing bijections $\sig_{\sigma}\cap\dd(\sig)\overset{\cong}\longrightarrow\dd(\sig/\sigma)$ and $\sig_{\sigma}\cap\ff(\sig)\overset{\cong}\longrightarrow\ff(\sig/\sigma)$. Moreover, for $k\in\Z$ it induces a bijection $\sig_{\sigma}\cap\sig_{k+\dim(\sigma)}\overset{\cong}\longrightarrow(\sig/\sigma)_k$. If $\sig$ is equifulldimensional or complete, then so is $\sig/\sigma$.

A $W$-semifan $\sig$ is complete if and only if $\sig/s(\sig)$ is complete. Moreover, if $\sig\neq\emptyset$ then there is a bijection between the set of $W$-completions of $\sig$ and the set of $W_{s(\sig)}$-completions of $\sig/s(\sig)$.
\end{no}

Next, we introduce a notion of direct sums of polycones. It may seem naive to define this to be a sum of polycones such that the sum of the vector spaces generated by these polycones is direct. However, this idea leads to a well-behaved notion of decomposition and turns out to be useful for different purposes in later sections. It should be noted that our notion of indecomposability differs from the ones in \cite{smi} (based on Minkowski sums) and \cite{wlo} (involving regularity conditions).\medskip

\begin{no}\label{dec10}
Let $(\sigma_i)_{i\in I}$ be a family of $W$-polycones. It is readily checked that the sum of $\R$-vector spaces $\sum_{i\in I}\langle\sigma_i\rangle$ is direct if and only if every $x\in\sum_{i\in I}\sigma_i$ can be written uniquely in the form $x=\sum_{i\in I}x_i$ with $x_i\in\sigma_i$ for every $i\in I$. Then, by abuse of language we say that \textit{the sum (of $W$-polycones) $\sum_{i\in I}\sigma_i$ is direct} and denote it by $\bigoplus_{i\in I}\sigma_i$. A $W$-polycone $\sigma$ is called \textit{$W$-decomposable} if it is the direct sum of two $W$-polycones different from $0$, and \textit{$W$-indecomposable} otherwise. If $\sigma$ is a $W$-polycone, then a \textit{$W$-decomposition of $\sigma$} is a set $Z$ of $W$-indecomposable $W$-polycones different from $0$ such that $\sigma=\bigoplus_{\tau\in Z}\tau$.
\end{no}

\begin{proposition}\label{dec20}
Let $(\sigma_i)_{i\in I}$ be a family of $W$-polycones such that the sum $\sigma\dfgl\sum_{i\in I}\sigma_i$ is direct.

{\rm a)} $\face(\sigma)=\{\bigoplus_{i\in I}\tau_i\mid\forall i\in I:\tau_i\fleq\sigma_i\}$.

{\rm b)} If $\tau_i\fleq\sigma_i$ for every $i\in I$ then $\tau_i=(\bigoplus_{j\in I}\tau_j)\cap\sigma_i$ for every $i\in I$.

{\rm c)} If $\tau\fleq\sigma$ then $\tau\cap\sigma_i\fleq\sigma$ for every $i\in I$ and $\tau=\bigoplus_{i\in I}\tau\cap\sigma_i$.
\end{proposition}

\begin{proof}
We can suppose that $\sigma$ is full. Let $V_i\dfgl\langle\sigma_i\rangle$ for $i\in I$ and let $\tau\subseteq V$. As $V^*=\bigoplus_{i\in I}V_i^*$ we have $\tau\fleq\sigma$ if and only if for every $i\in I$ there is a $u_i\in V_i^*$ with $\sigma\subseteq(\sum_{i\in I}u_i)^{\vee,V}$ and $\tau=\sigma\cap(\sum_{i\in I}u_i)^{\perp,V}$, and it is straightforward to check that these conditions are equivalent to $\sigma_i\subseteq u_i^{\vee,V}$ for every $i\in I$ and $\tau=\sum_{i\in I}\sigma_i\cap u_i^{\perp,V_i}$. This holds if and only if for every $i\in I$ there is a $\tau_i\fleq\sigma_i$ with $\tau=\sum_{i\in I}\tau_i$. This sum being obviously direct claim a) is proven. Now, b) and c) follow easily.
\end{proof}

\begin{no}\label{dec25}
Let $(\sigma_i)_{i\in I}$ be a family of $W$-polycones such that the sum $\sigma\dfgl\sum_{i\in I}\sigma_i$ is direct. Then, $\sigma$ is sharp if and only if $\sigma_i$ is sharp for every $i\in I$, and in case $\card(I)\neq 1$ this holds if and only if $\sigma_i\fleq\sigma$ for every $i\in I$ (\ref{dec20}, \ref{topprop}~b)).
\end{no}

\begin{lemma}\label{partition}
{\rm a)} Let $\sigma$ be a sharp $W$-polycone, let $Z$ be a $W$-decomposition of $\sigma$, let $u_{\rho}\in\rho\setminus 0$ for $\rho\in\sigma_1$, let $U\dfgl\{u_{\rho}\mid\rho\in\sigma_1\}$, and for $\tau\in Z$ let $U_{\tau}\dfgl\{u_{\rho}\mid\rho\in\tau_1\}$. Then, $(U_{\tau})_{\tau\in Z}$ is a partition of $U$ and the sum of $\R$-vector spaces $\sum_{\tau\in Z}\langle U_{\tau}\rangle$ is direct. Moreover, if $\tau\in Z$ then every partition $(C_l)_{l\in L}$ of $U_{\tau}$ such that the sum of $\R$-vector spaces $\sum_{l\in L}\langle C_l\rangle$ is direct, has cardinality $1$.

{\rm b)} Let $U\subseteq V\setminus 0$ be finite. There is at most one partition $(A_p)_{p\in P}$ of $U$ such that the sum of $\R$-vector spaces $\sum_{p\in P}\langle A_p\rangle$ is direct and with the property that if $p\in P$ then every partition $(C_l)_{l\in L}$ of $A_p$ such that the sum of $\R$-vector spaces $\sum_{l\in L}\langle C_l\rangle$ is direct, has cardinality $1$.
\end{lemma}

\begin{proof}
a) The definition of $U_{\tau}$ makes sense, and for $\rho\in\sigma_1$ there is a unique $\tau\in Z$ with $\rho\in\tau_1$ (\ref{dec20}). If $\tau\in Z$ then $\tau_1\neq\emptyset$ and $\langle U_{\tau}\rangle=\langle\tau\rangle$ (\ref{topprop}~b)), so that $(\tau_1)_{\tau\in Z}$ is a partition of $\sigma_1$. Hence, $(U_{\tau})_{\tau\in Z}$ is a partition of $U$ and the sum $\sum_{\tau\in Z}\langle U_{\tau}\rangle$ is direct. If $\tau\in Z$ and $(C_l)_{l\in L}$ is a partition of $U_{\tau}$ such that the sum $\sum_{l\in L}\langle C_l\rangle$ is direct, then $\{\cone(C_l)\mid l\in L\}$ is a $W$-decomposition of $\tau$. Thus, $W$-indecomposability of $\tau$ implies $\card(L)=1$, hence the claim.

b) Let $(A_p)_{p\in P}$ and $(B_q)_{q\in Q}$ be partitions of $U$ as in the claim, and let $q\in Q$. For $u\in B_q$ there is a unique $p_u\in P$ with $u\in A_{p_u}$, so that $\{B_q\cap A_{p_u}\mid u\in B_q\}$ is a partition of $B_q$ and the sum $\sum\{\langle B_q\cap A_{p_u}\rangle\mid u\in B_q\}$ is direct. Therefore, $\{B_q\cap A_{p_u}\mid u\in B_q\}$ has cardinality $1$, so there is a unique $p\in P$ with $B_q\subseteq A_p$. By reasons of symmetry we get $\{A_p\mid p\in P\}=\{B_q\mid q\in Q\}$, thus the claim.
\end{proof}

\begin{theorem}\label{dec50}
Every $W$-polycone has a $W$-decomposition, and a $W$-poly\-cone has a unique $W$-decomposition if and only if it is sharp or a line.
\end{theorem}

\begin{proof}
Let $\sigma$ be a $W$-polycone. Existence of a $W$-decomposition of $\sigma$ follows by induction on $\dim(\sigma)$, since $\dim(\sigma)\leq 1$ obviously implies that $\sigma$ is $W$-in\-de\-com\-po\-sable. If $p:V\twoheadrightarrow V/s(\sigma)$ denotes the canonical epimorphism then we can choose a $(W/s(\sigma),W)$-rational section $q$ of $p$ (\ref{compl}) and a basis $X\subseteq W$ of $s(\sigma)$. Then, $\sigma=\sum_{x\in X}\langle x\rangle+q(p(\sigma))$, and this sum is direct. Therefore, uniqueness of a $W$-decomposition of $\sigma$ implies that it is either sharp or a line. Conversely, if $\sigma$ is a line then it obviously has a unique $W$-decomposition. So, suppose that $\sigma$ is sharp, let $Z$ and $Z'$ be $W$-decompositions of $\sigma$, let $u_{\rho}\in\rho\setminus 0$ for $\rho\in\sigma_1$, and let $U\dfgl\{u_{\rho}\mid \rho\in\sigma_1\}$. Then, the partitions $(\{u_{\rho}\mid\rho\in\tau_1\})_{\tau\in Z}$ and $(\{u_{\rho}\mid\rho\in\tau_1\})_{\tau\in Z'}$ of $U$ coincide (\ref{partition}), and this implies $Z=Z'$.
\end{proof}

\begin{no}
It follows from \ref{dec50} that $W$-indecomposability and $V$-indecomposa\-bi\-li\-ty of $W$-polycones are equivalent and that $W$-decompositions and $V$-decompositions of a sharp $W$-polycone are the same. Note that the second statement need not hold for nonsharp $W$-polycones.
\end{no}

\begin{no}\label{dec70}
A $W$-polycone $\sigma$ is simplicial if and only if it is sharp and $\sigma_1$ is a $W$-decomposition of $\sigma$. Hence, a simplicial $W$-polycone $\sigma$ is $W$-indecomposable if and only if $\dim(\sigma)\leq 1$. Moreover, if $(\sigma_i)_{i\in I}$ is a family of $W$-polycones such that the sum $\sigma\dfgl\sum_{i\in I}\sigma_i$ is direct, then $\sigma$ is simplicial if and only if $\sigma_i$ is simplicial for every $i\in I$.
\end{no}

To illustrate the above we sketch how it can be used to prove the well-known fact that every fan has a simplicial strict subdivision.\medskip

\begin{no}\label{sub10}
Let $\sig$ be a $W$-fan. A \textit{(strict) $W$-subdivision of $\sig$} is a $W$-fan $\sig'$ with $|\sig|\subseteq|\sig'|$ such that for every $\sigma\in\sig'$ there is a $\tau\in\sig$ with $\sigma\subseteq\tau$ (and that in addition $\sig'_1\subseteq\sig_1$). If $\Tau\subseteq\sig$ is a $W$-subfan and $\sig'$ is a (strict) $W$-subdivision of $\sig$, then $\{\sigma\in\sig'\mid\exists\tau\in\Tau:\sigma\subseteq\tau\}$ is a (strict) $W$-subdivision of $\Tau$, called \textit{the $W$-subdivision of $\Tau$ induced by $\sig'$.}

For $\rho\in\sig_1$ we set $\sig[\rho]\dfgl\{\rho+\tau\mid\exists\sigma\in\sig_{\rho}:\rho\not\fleq\tau\fleq\sigma\}=\{\rho\oplus\tau\mid\exists\sigma\in\sig_{\rho}:\rho\not\fleq\tau\fleq\sigma\}$. On use of \ref{dec20} it can be shown that $\sig(\rho)\dfgl(\sig\setminus\sig_{\rho})\cup\sig[\rho]$ is a strict $W$-subdivision of $\sig$. Now, we choose a counting $(\rho^{(i)})_{i=1}^r$ of $\sig_1$ and set $\sig^{(0)}\dfgl\sig$ and $\sig^{(i)}\dfgl\sig^{(i-1)}(\rho^{(i)})$ for $i\in[1,r]$. Then, $\sig^{(r)}$ is a strict $W$-subdivision of $\sig$ by the above. If it is nonsimplicial then there is a $W$-indecomposable $\sigma\in\sig^{(r)}$ with $\dim(\sigma)>1$ (\ref{dec50}, \ref{dec70}). Therefore, $m\dfgl\max\{i\in[1,r]\mid\rho^{(i)}\fleq\sigma\}$ exists and $\sigma\in\sig^{(m-1)}[\rho^{(m)}]$, hence $\sigma=\rho^{(m)}\oplus\tau$ for a $\tau\in\sig^{(m-1)}$. Indecomposability of $\sigma$ implies $\tau=0$, hence the contradiction $\sigma=\rho^{(m)}\in\sig_1$. Thus, $\sig^{(r)}$ is simplicial.
\end{no}


\section{The frontier of a fan}\label{sec3}

\textit{Throughout this section let $\sig$ be a $W$-semifan.}\medskip

The goal of this section is a combinatorial description of the topological frontier of the support of $\sig$ (\ref{top80}). A key step for this is a characterisation of completeness of $\sig$ in terms of its projections (\ref{top50}).\medskip

\begin{no}\label{top10}
For $x\in|\sig|$ we denote by $\omega_{x,\sig}\dfgl\bigcap\{\sigma\in\sig\mid x\in\sigma\}$ (or just by $\omega_x$) the smallest cone in $\sig$ containing $x$. For $\sigma\in\sig$ we have $x\in\inn_{\langle\sigma\rangle}(\sigma)$ if and only if $\sigma=\omega_x$ (\ref{topprop}~a)).
\end{no}

\begin{proposition}\label{top20}
If $x\in\fr(|\sig|)$ then $\omega_x\subseteq\fr(|\sig|)$.
\end{proposition}

\begin{proof}
As $\omega_x=\cl_{\langle\omega_x\rangle}(\inn_{\langle\omega_x\rangle}(\omega_x))$ (\ref{topprop}~a)) it suffices to show $\inn_{\langle\omega_x\rangle}(\omega_x)\subseteq\fr(|\sig|)$. Let $y\in\inn_{\langle\omega_x\rangle}(\omega_x)\setminus\{x\}$ and let $U\in\V(y)$. We have to show $U\not\subseteq|\sig|$. As $\bigcup\{\sigma\in\sig\mid y\notin\sigma\}$ is closed and does not contain $y$ we can suppose every cone in $\sig$ met by $U$ to contain $y$, and as $V$ is locally convex we can suppose $U$ to be convex. Let $L$ be the affine line through $x$ and $y$. There is a $z\in L\cap\omega_x$ with $x\in\,\strecker z,y\strecker$. Convexity of $U$ yields $C\dfgl\conv(\{z\}\cup U)=\bigcup_{q\in U}\streckel z,q\strecker$, and it is readily checked that $x\in\inn(C)$. As $x\in\fr(|\sig|)$ there is a $p\in C\setminus|\sig|$, hence a $q\in U$ with $p\in\streckel z,q\strecker$. If $q\in|\sig|$ then there is a $\sigma\in\sig$ with $q\in\sigma$, implying $y\in\sigma\cap\inn_{\langle\omega_x\rangle}(\omega_x)$, hence $z\in\omega_x=\omega_y\fleq\sigma$ (\ref{top10}) and thus the contradiction $p\in\streckel z,q\strecker\subseteq\sigma\subseteq|\sig|$. This proves the claim.
\end{proof}

\begin{lemma}\label{top30}
Let $x\in\fr(|\sig|)$, let $U\in\V(x)$ be convex, and let $A\subseteq V$ be closed and nowhere dense in $V$ with $x\in A$ and containing for every $y\in A\setminus\{x\}$ the affine line through $x$ and $y$. There exists a nonempty open subset $U'\subseteq U$ such that for every $y\in U'$ we have $\,\strecker x,y\strecker\cap(|\sig|\cup A)=\emptyset$.
\end{lemma}

\begin{proof}
First we can suppose that every cone in $\sig$ met by $U$ contains $x$, and then we can suppose that every $\sigma\in\sig_{<n}$ containing $x$ is contained in $A$ (\ref{topprop}~a)). As $x\in\fr(|\sig|)$ and $U\in\V(x)$, closedness of $|\sig|$ implies that $U\setminus|\sig|$ contains a nonempty, open set $U''$. As $A$ is nowhere dense in $V$ (\ref{topprop}~a)) and $|\sig|\cup A$ is closed, $U''\setminus(|\sig|\cup A)$ contains a nonempty, open set $U'$. Let $y\in U'$, and assume that $\,\strecker x,y\strecker$ meets $|\sig|\cup A$. If $\,\strecker x,y\strecker$ meets $A$, then $y$ lies on the affine line through $x$ and a point in $A$, thus yielding the contradiction $y\in A$. Thus, $\,\strecker x,y\strecker$ does not meet $A$, and so it meets a $\sigma\in\sig_n$, hence $\fr(\sigma)$, and therefore a $\tau\in\pface(\sigma)$ (\ref{topprop}~a)). It follows $x\notin\tau$, hence $\tau\cap U=\emptyset$, contradictory to $\,\strecker x,y\strecker\subseteq U$, and herewith the claim is proven.
\end{proof}

\begin{proposition}\label{top40}
If $\tau\fleq\sigma\in\sig$ and $\sigma\subseteq\fr(|\sig|)$ then $\sigma/\tau\subseteq\fr(|\sig/\tau|)$.
\end{proposition}

\begin{proof}
As continuity of $p_{s(\sig)}$ implies $p_{s(\sig)}(\fr(|\sig|))\subseteq\fr(|\sig/s(\sig)|)$ we can suppose $\sig$ to be a $W$-fan. We prove the claim by induction on $d\dfgl\dim(\tau)$. If $d=0$ it is clear. So, let $d=1$ and hence $n>1$ (\ref{topprop}~a)). By replacing $\sig$ with $\bigcup_{\omega\in\sig_{\tau}}\face(\omega)$ we can suppose that $\sig_{\max}\subseteq\sig_{\tau}$. By continuity of $p_{\tau}$ and \ref{topprop}~a) it suffices to show $p_{\tau}(\inn_{\langle\sigma\rangle}(\sigma))\subseteq\fr(|\sig/\tau|)$. So, let $x\in\inn_{\langle\sigma\rangle}(\sigma)$, hence $\sigma=\omega_x\subseteq\fr(|\sig|)$ (\ref{top10}) and therefore $\dim(\sigma)<n$ (\ref{topprop}~a)). Let $U\in\V_{V_{\tau}}(p_{\tau}(x))$. We have to show that $U$ meets $V_{\tau}\setminus|\sig/\tau|$. As $V$ is locally convex we can suppose $U$ to be convex, hence $p_{\tau}^{-1}(U)\in\V_V(x)$ is convex, too. Therefore, by \ref{top30} there is a $y\in p_{\tau}^{-1}(U)\setminus\{x\}$ with $\,\strecker x,y\strecker\subseteq p_{\tau}^{-1}(U)\setminus(|\sig|\cup\langle\sigma\rangle)$.

Let $\omega\in\sig_{\max}$. There is a hyperplane $H\subseteq V$ separating $\sigma$ and $\omega$ in their intersection (\ref{sep}), and we denote the halfspaces defined by $H$ and containing $\sigma$ and $\omega$ by $H_{\sigma}$ and $H_{\omega}$, respectively. We will show that there is a $y_{\omega}\in\,\strecker x,y\strecker$ with $\omega\cap(\tau+\cone(x,y_{\omega}))\subseteq\sigma$. Indeed, if $\omega\cap(\tau+\cone(x,y))\subseteq\sigma$ then $y_{\omega}\dfgl y$ fulfils the claim. So, suppose $\omega\cap(\tau+\cone(x,y))\not\subseteq\sigma$. First, we assume that $x\in\omega$ and hence $\sigma=\omega_x\fleq\omega$. Let $z\in\tau\setminus 0$ and let $w\in\omega\cap(\tau+\cone(x,y))\setminus\sigma$. As $x\in\inn_{\langle\sigma\rangle}(\sigma)$ there is an $r\in\R_{>0}$ with $x+r(x-z)\in\inn_{\langle\sigma\rangle}(\sigma)\subseteq\omega$. As $w\in\tau+\cone(x,y)=\cone(z,x,y)$ there is an $r\in\R_{\geq 0}$ with $sw\in\omega\cap\conv(z,x,y)$. Now it is straightforward to derive the contradiction $$\emptyset\neq\,\strecker x+(r-z),sw\streckel\,\cap\,\strecker x,y\streckel\,\subseteq\omega\cap\,\strecker x,y\strecker\subseteq|\sig|\cap\,\strecker x,y\strecker=\emptyset.$$ Thus, $x\in\sigma\setminus\omega$ and in particular $x\in H_{\sigma}\setminus H$. Furthermore, as $y\in H_{\sigma}$ implies the contradiction $\omega\cap(\tau+\cone(x,y))\subseteq\omega\cap H_{\sigma}\subseteq\sigma$, we have $y\in H_{\omega}\setminus H$. Hence there is a $y_{\omega}\in\,\strecker x,y\streckel\,\cap H$, yielding $\omega\cap(\tau+\cone(x,y_{\omega}))\subseteq\omega\cap H_{\sigma}\subseteq\sigma$ as desired.

The above being done for every $\omega\in\sig_{\max}$ we see that there is a $y'\in\,\strecker x,y\strecker$ such that if $\omega\in\sig_{\max}$ then $\omega\cap(\tau+\cone(x,y'))\subseteq\sigma$. As $\,\strecker x,y\strecker\cap\langle\sigma\rangle=\emptyset$ we get $(\tau+\cone(y'))\cap\sigma=\tau$, hence $\tau+y'\subseteq(\tau+\cone(x,y'))\setminus\sigma$, and therefore $(\tau+y')\cap\omega\subseteq(\tau+\cone(x,y'))\cap\omega\setminus\sigma=\emptyset$ for every $\omega\setminus\sig_{\max}$. Thus, $(\tau+y')\cap|\sig|=\emptyset$. If $\langle\tau\rangle+y'$ meets $|\sig|$ then there are $\omega\in\sig_{\max}\subseteq\sig_{\tau}$ with $y'\in\omega+\tau\subseteq\omega$, hence $\tau+y'\subseteq\tau+\omega\subseteq\omega$, yielding the contradiction $y'+\tau=\emptyset$. Thus $(\langle\tau\rangle+y')\cap|\sig|=\emptyset$, and it is readily checked that $p_{\tau}(y')\in U\setminus|\sig/\tau|$. Herewith the claim is proven in case $d=1$.

Finally, if $d>1$ and the claim holds for strictly smaller values of $d$, then for every $\rho\in\tau_1$ we have $\tau/\rho\fleq\sigma/\rho\subseteq\fr(|\sig/\rho|)$ (\ref{proj}), hence $\sigma/\tau=(\sigma/\rho)/(\tau/\rho)\subseteq\fr(|(\sig/\rho)/(\tau/\rho)|)=\fr(|\sig/\tau|)$, and so the claim follows by \ref{topprop}~b).
\end{proof}

\begin{theorem}\label{top50}
If $\dim(s(\sig))\neq n-1$ and $\sig_{\dim(s(\sig))+1}\neq\emptyset$, then $\sig$ is complete if and only if $\sig/\sigma$ is complete for every $\sigma\in\sig_{\dim(s(\sig))+1}$.
\end{theorem}

\begin{proof}
We can suppose that $\sig$ is a $W$-fan (\ref{proj}), hence $n\neq 1$. If $\sig$ is complete then so is $\sig/\sigma$ for every $\sigma\in\sig$ (\ref{proj}). Conversely, suppose that $\sig/\sigma$ is complete for every $\sigma\in\sig_1$ and assume that $\sig$ is noncomplete. If $x\in\fr(|\sig|)\setminus 0$ then $\omega_x\subseteq\fr(|\sig|)$ (\ref{top20}), and there is a $\rho\in(\omega_x)_1$ (\ref{topprop}~b)), implying $\omega_x/\rho\subseteq\fr(|\sig/\rho|)$ (\ref{top40}) and thus the contradiction that $\sig/\rho$ is noncomplete. So, $\fr(|\sig|)=0$ and there is an $x\in V\setminus|\sig|$. As $\sig_1\neq\emptyset$ we have $\sig_n\neq\emptyset$ (\ref{proj}), so that there is a $y\in\inn(|\sig|)$ (\ref{topprop}~a)), hence a $U\in\V(y)$ with $U\subseteq|\sig|$. If $z\in U$ then $\,\strecker z,x\streckel\,$ meets $\fr(|\sig|)=0$, hence $z\in\cone(-x)$. This shows $U\subseteq\cone(-x)$, hence the contradiction $n=1$, and thus the claim is proven.
\end{proof}

\begin{corollary}\label{top60}
If $\sig$ is equifulldimensional and $\ff(\sig)=\emptyset$, then $\sig$ is complete.
\end{corollary}

\begin{proof}
On use of \ref{proj} and \ref{top50} this follows easily by induction on $n$.
\end{proof}

\begin{proposition}\label{top70}
The following statements are equivalent:
\begin{equi}
\item[\rm(i)] $\sig$ is equifulldimensional or empty;
\item[\rm(ii)] $\fr(|\sig|)=\bigcup\ff(\sig)$;
\item[\rm(iii)] $\cl(\inn(|\sig|))=|\sig|$.
\end{equi}
\end{proposition}

\begin{proof}
``(i)$\Rightarrow$(ii)'': If $x\in\fr(|\sig|)$ then $\omega_x\subseteq\fr(|\sig|)$ (\ref{top20}), so $\sig/\omega_x$ is not complete (\ref{top40}). This implies $\ff(\sig/\omega_x)\neq\emptyset$ (\ref{top60}), hence $\sig_{\omega_x}\cap\ff(\sig)\neq\emptyset$ (\ref{proj}) and thus $x\in\bigcup\ff(\sig)$. Conversely, let $\sigma\in\ff(\sig)$ and $\tau\in\sig_n$ with $\sigma\flneq\tau$, and assume $\sigma\not\subseteq\fr(|\sig|)$. Then $\sigma$ meets $\inn(|\sig|)$, hence there is a $y\in\inn_{\langle\sigma\rangle}(\sigma)\cap\inn(|\sig|)$ (\ref{topprop}~a)), implying $y\notin\bigcup\sig_{\geq n-1}\setminus\{\sigma,\tau\}$ (\ref{top10}). So, there is an open $U\in\V(y)$ with $U\subseteq|\sig|$ and $U\cap\langle\sigma\rangle\subseteq\sigma$ such that the only cones in $\sig_{\geq n-1}$ met by $U$ are $\tau$ and $\sigma$. Equifulldimensionality of $\sig$ implies $U\subseteq\tau$, hence $U\subseteq\inn(\tau)$ and therefore the contradiction $y\in\tau\setminus\sigma$ (\ref{topprop}~a)). This shows $\sigma\subseteq\fr(|\sig|)$.

``(ii)$\Rightarrow$(iii)'': For $\sigma\in\ff(\sig)$ there is a $\tau\in\sig_n$ with $\sigma\flneq\tau$, hence $\sigma\subseteq\fr(\tau)=\fr(\inn(\tau))$ (\ref{topprop}~a)). As $\inn(\tau)\subseteq\inn(|\sig|)$ this shows $\sigma\subseteq\cl(\inn(|\sig|))$. It follows $\fr(|\sig|)=\bigcup\ff(\sig)\subseteq\cl(\inn(|\sig|))$.

``(iii)$\Rightarrow$(i)'': As $|\sig|=(\bigcup\dd(\sig))\cup(\bigcup\sig_n)$ and $\dd(\sig)$ is nowhere dense in $V$ (\ref{topprop}~a)) we have $\inn(|\sig|)=\inn(\bigcup\sig_n)$, hence $|\sig|=\cl(\inn(\bigcup\sig_n))\subseteq\bigcup\sig_n\subseteq|\sig|$ and therefore $|\sig|=\bigcup\sig_n$. So, every $\sigma\in\dd(\sig)$ is covered by the family $(\tau\cap\sigma)_{\tau\in\sig_n}$ in $\pface(\sigma)$ and hence nowhere dense in $\langle\sigma\rangle$, implying $\dd(\sig)=\emptyset$ (\ref{topprop}~a)).
\end{proof}

\begin{theorem}\label{top80}
$\fr(|\sig|)=(\bigcup\dd(\sig))\cup(\bigcup\ff(\sig))$.
\end{theorem}

\begin{proof}
As $|\sig|=(\bigcup\dd(\sig))\cup(\bigcup\sig_n)$ and $\dd(\sig)$ is closed and nowhere dense in $V$ (\ref{topprop}~a)) we have $\fr(|\sig|)=(\bigcup\dd(\sig))\cup\fr(\bigcup\sig_n)$. As the $W$-subsemifan $\sig'\dfgl\bigcup_{\sigma\in\sig_n}\face(\sigma)$ of $\sig$ is equifulldimensional or empty with $\ff(\sig')=\ff(\sig)$, the claim follows from \ref{top70}.
\end{proof}

\begin{corollary}\label{top100}
a) If $\tau\fleq\sigma\in\sig$, then $\sigma\subseteq\fr(|\sig|)$ if and only if $\sigma/\tau\subseteq\fr(|\sig/\tau|)$.

b) Every $\sigma\in\sig$ with $\sigma\subseteq\inn(|\sig|)$ is the intersection of a family in $\sig_n$.
\end{corollary}

\begin{proof}
a) follows easily from \ref{top80} and \ref{proj}. b) If $\dim(\sigma)\geq n-1$ this is clear by \ref{top80}. Otherwise $\sigma$ is a face of a full cone in $\sig$ (\ref{top80}) and then the claim follows from \ref{topprop}~b).
\end{proof}


\section{Packings and strong completions}\label{sec5}

\textit{Throughout this section let $\sig$ be a $W$-fan.}\medskip

We introduce three properties of an extension $\sig\subseteq\sig'$. First, \textit{relative simpliciality} means that $\sig'$ is ``as simplicial as possible''; this will be needed to construct a completion that is simplicial in case $\sig$ is so. Second, \textit{separability} means that $\sig'$ is ``as independent as possible'' from $\sig$; this will allow us to make certain changes to $\sig'$ without changing $\sig$. Third, \textit{tight separability} strengthens separability by ensuring that $\sig'$ is on one hand independent of $\sig$ (that is, separable), but on the other hand not too much bigger than $\sig$ (in some topological sense). It is at this point where in \cite{ei} the metric $\eps$-arguments enter the scene.\medskip

\begin{no}
By abuse of language, for a $W$-polycone $\sigma$ we set $\sigma\cap\sig\dfgl\{\sigma\cap\tau\mid\tau\in\sig\}$, and $\sigma$ is called \textit{free over $\sig$} if $\sigma\cap\sig\subseteq\{0\}$. If $\sigma$ is free over $\sig$ then so is every face of $\sigma$.  If $\sig'$ is a $W$-extension of $\sig$ with $\sigma\in\sig'$ and $(\sigma_i)_{i\in I}$ is a family in $\sig'$ with $\sigma=\bigoplus_{i\in I}\sigma_i$, then $\sigma$ is free over $\sig$ if and only if $\sigma_i$ is free over $\sig$ for every $i\in I$ (\ref{dec20}). A set of $W$-polycones is called \textit{free over $\sig$} if all its elements are so. If $\sig'$ is a $W$-fan that is free over $\sig$, then every $W$-subdivision of $\sig'$ is free over every $W$-subfan of $\sig$. 

A $W$-extension $\sig'$ of $\sig$ is called \textit{relatively simplicial (over $\sig$)} if every cone in $\sig'$ that is free over $\sig$ is simplicial. Clearly, $\sig$ is relatively simplicial over itself, and simplicial $W$-extensions of $\sig$ are relatively simplicial. Conversely, if $\sig$ is simplicial then a relatively simplicial $W$-extension of $\sig$ is not necessarily simplicial.

Let $\sig\subseteq\sig'\subseteq\sig''$ be $W$-extensions. If $\sig\subseteq\sig''$ is relatively simplicial then so are $\sig\subseteq\sig'$ and $\sig'\subseteq\sig''$, but the converse does not necessarily hold.
\end{no}

\begin{example}
If $n=3$, then the facial fan of a nonsimplicial $W$-polycone is relatively simplicial over all its $W$-subfans except $\{0\}$ and $\emptyset$.
\end{example}

\begin{no}\label{pack20}
Let $\sigma$ be a $W$-polycone. If there is a pair $(\tau,\tau')$ of $W$-polycones such that $\sigma=\tau\oplus\tau'$, that $\tau\in\sig\cup\{0\}$ and that $\tau'$ is sharp and free over $\sig$, then $\sigma$ is sharp, $\tau_1=\sigma_1\cap\sig$ and $\tau'_1=\sigma_1\setminus\sig$ (\ref{topprop}~b), \ref{dec20}), so that there is at most one such pair (\ref{topprop}~b)). If such a pair exists then $\sigma$ is called \textit{separable over $\sig$} and we set $\inn_{\sig}(\sigma)\dfgl\tau$ and $\exte_{\sig}(\sigma)\dfgl\tau'$.

Let $\sigma$ be separable over $\sig$. If $\tau\fleq\sigma$ then $\tau$ is separable over $\sig$ with $\inn_{\sig}(\tau)=\tau\cap\inn_{\sig}(\sigma)\fleq\inn_{\sig}(\sigma)$ and $\exte_{\sig}(\tau)=\tau\cap\exte_{\sig}(\sigma)\fleq\exte_{\sig}(\sigma)$ (\ref{dec20}). If $\omega$ is a further $W$-polycone that is separable over $\sig$ such that $\sigma\cap\omega\in\face(\sigma)\cap\face(\omega)$, then by the above $\sigma\cap\omega$ is separable over $\sig$ with $\inn_{\sig}(\sigma\cap\omega)=\inn_{\sig}(\sigma)\cap\inn_{\sig}(\omega)$ and $\exte_{\sig}(\sigma\cap\omega)=\exte_{\sig}(\sigma)\cap\exte_{\sig}(\omega)$.

A $W$-extension $\sig'$ of $\sig$ is called \textit{separable (over $\sig$)} if every cone in $\sig'$ is separable over $\sig$, and then we set $\exte_{\sig}(\sig')\dfgl\{\exte_{\sig}(\sigma)\mid\sigma\in\sig'\}$. Clearly, $\sig$ is separable over itself. If $\sig$ is simplicial then so is every separable, relatively simplicial $W$-extension of $\sig$ (\ref{dec70}).

Let $\sig\subseteq\sig'\subseteq\sig''$ be $W$-extensions. If $\sig\subseteq\sig'$ and $\sig'\subseteq\sig''$ are separable then so is $\sig\subseteq\sig''$, and if $\sigma\in\sig''$ then $\inn_{\sig}(\sigma)=\inn_{\sig}(\inn_{\sig'}(\sigma))$ and $\exte_{\sig}(\sigma)=\exte_{\sig}(\inn_{\sig'}(\sigma))\oplus\exte_{\sig'}(\sigma)$. Conversely, if $\sig\subseteq\sig''$ is separable then so is $\sig\subseteq\sig'$, but $\sig'\subseteq\sig''$ is not necessarily so.
\end{no}

\begin{no}
A separable $W$-extension $\sig'$ of $\sig$ is called \textit{tightly separable (over $\sig$)} if for every $\sigma\in\sig'$ we have $\exte_{\sig}(\sigma)\setminus 0\subseteq\fr_{\langle|\sig'|\rangle}(|\sig'|)$. Clearly, $\sig$ is tightly separable over itself.

Let $\sig\subseteq\sig'\subseteq\sig''$ be $W$-extensions. If $\sig\subseteq\sig''$ is tightly separable and $\langle|\sig'|\rangle=\langle|\sig''|\rangle$, then $\sig\subseteq\sig'$ is tightly separable. If $\sig\subseteq\sig''$ is tightly separable and $\sig'\subseteq\sig''$ is separable, then $\sig'\subseteq\sig''$ is tightly separable. Conversely, if $\sig\subseteq\sig'$ and $\sig'\subseteq\sig''$ are tightly separable then $\sig\subseteq\sig''$ is not necessarily so, even if $\langle|\sig|\rangle=\langle|\sig''|\rangle$.
\end{no}

\begin{example}
For $i\in[1,3]$ we denote by $\sigma^{(i)}$ the $i$-th quadrant in $\R^2$. Then, the $\Q^2$-fan $\bigcup_{i=1}^3\face(\sigma^{(i)})$ is tightly separable over $\face(\sigma^{(2)})$, but not over $\face(\sigma^{(1)})$.
\end{example}

Finally, we put the above notions together to define what we are ultimately going to construct.\medskip

\begin{no}
A \textit{$W$-quasipacking of $\sig$} is a $W$-extension $\sig\subseteq\sig'$ such that $|\sig|\setminus 0\subseteq\inn(|\sig'|)$. A \textit{$W$-packing of $\sig$} is a relatively simplicial, tightly separable $W$-quasipacking $\sig\subseteq\sig'$ such that $\sig'_1$ is empty or $\sig'$ is equifulldimensional. A \textit{strong $W$-completion of $\sig$} is a pair $(\sigquer,\sighut)$ such that $\sigquer$ is a $W$-packing of $\sig$ and $\sighut$ is a $W$-completion of $\sigquer$ that is relatively simplicial over $\sig$. Finally, if $\sig\subseteq\sig'$ is a $W$-extension then we denote by $C_{\sig}(\sig')$ and $c_{\sig}(\sig')$ the set $\{\rho\in\sig_1\mid\sig'/\rho\text{ is noncomplete}\}$ and its cardinality, respectively.
\end{no}

\begin{examples}\label{scexas}
a) $\sig$ is a $W$-packing of itself if and only if it is a $W$-quasipacking of itself (\ref{topprop}~a)), and this holds if and only if $\sig$ is complete, or $\sig_1=\emptyset$, or $n=1$.

b) If $\sig\subseteq\sigquer$ is a $W$-packing then $(\sigquer,\sigquer)$ is a strong $W$-completion of $\sig$ if and only if $\sigquer$ is complete. In particular, $(\sig,\sig)$ is a strong $W$-completion of $\sig$ if and only if $\sig$ is complete.

c) If $\sig_1=\emptyset$ and $\Omega$ is a complete, simplicial $W$-fan then $(\sig,\Omega)$ is a strong $W$-completion of $\sig$.

d) If $n=1$, $x\in W\setminus 0$, $\sig\dfgl\{0,\cone(x)\}$ and $\sighut\dfgl\{0,\cone(x),\cone(-x)\}$, then $(\sig,\sighut)$ is a strong $W$-completion of $\sig$.
\end{examples}

\begin{proposition}\label{qpc}
Let $\sig\subseteq\sig'$ be a $W$-extension. Then, $C_{\sig}(\sig')=\{\rho\in\sig_1\mid\rho\subseteq\fr(|\sig'|)\}$, and $\sig'$ is a $W$-quasipacking of $\sig$ if and only if $c_{\sig}(\sig')=0$.
\end{proposition}

\begin{proof}
The first statement follows immediately from \ref{top80} and \ref{proj}. Hence, if $\sig'$ is a $W$-quasipacking of $\sig$ then $c_{\sig}(\sig')=0$. Conversely, assume that $c_{\sig}(\sig')=0$ and $|\sig|\setminus 0\not\subseteq\inn(|\sig'|)$. There is an $x\in|\sig|\setminus 0$ with $x\in\fr(|\sig'|)$, and it follows $0\neq\omega_{x,\sig}=\omega_{x,\sig'}\subseteq\fr(|\sig'|)$ (\ref{top20}), implying the contradiction that $\sig'/\rho$ is noncomplete for every $\rho\in(\omega_x)_1$ (\ref{top40}, \ref{topprop}~b)), and thus showing the second statement.
\end{proof}

\begin{proposition}\label{quasilemma}
If $\sig'$ is a separable $W$-quasipacking of $\sig$ then every $W$-exten\-sion of $\sig'$ is separable over $\sig$.
\end{proposition}

\begin{proof}
Let $\sig''$ be a $W$-extension of $\sig'$ and let $\sigma\in\sig''\setminus\sig'$. If $\sigma$ meets $\inn(|\sig|)$ then there are a $\tau\in\sig$ and an $x\in\inn_{\langle\sigma\rangle}(\sigma)\cap\tau$ (\ref{topprop}~a)), implying $x\in\sigma\cap\tau\flneq\sigma$ and hence the contradiction $x\in\fr_{\langle\sigma\rangle}(\sigma)$ (\ref{topprop}~a)). So, we get $\sigma\cap|\sig|\setminus 0\subseteq\sigma\cap\inn(|\sig|)=\emptyset$, implying that $\sigma$ is free over $\sig$ and thus the claim. 
\end{proof}

\begin{proposition}\label{fff}
{\rm a)} If $\sig'$ is an equifulldimensional, tightly separable $W$-exten\-sion of $\sig$ then $\exte_{\sig}(\sig')\subseteq\fff(\sig')$.

{\rm b)} If $\sig'$ is a $W$-quasipacking of $\sig$ then $\fff(\sig')\subseteq\exte_{\sig}(\sig')$.
\end{proposition}

\begin{proof}
This follows immediately from \ref{top70} and \ref{top80}.
\end{proof}


\section{Techniques for constructing extensions}\label{sec6}

\renewcommand{\thesubsection}{\Alph{subsection}}

\subsection{Constructing complete fans}\mbox{}\medskip

We give a general recipe for constructing a complete semifan from a given fan $\sig$ and some chosen additional data. However, this does not yield a completion of $\sig$ in general but induces under mild conditions a fan on $\cl(V\setminus|\sig|)$ that induces a subdivision of $\fff(\sig)$.\medskip

\begin{no}\label{comp10}
Let $H$ be a finite set of $W^*$-rational lines in $V^*$ and let $u=(u_L)_{L\in H}\in\prod_{L\in H}(L\setminus 0)$. The set of $W$-polycones $$\textstyle\{(\bigcap_{L\in U}u_L^{\vee})\cap(\bigcap_{L\in H\setminus U}-u_L^{\vee})\mid U\subseteq H\}$$ does not depend on $u$ but only on $H$; we denote it by $\Omega_H$ and set $\omegaq_H\dfgl\bigcup_{\sigma\in\Omega_H}\face(\sigma)$. This is a $W$-semifan. Indeed, it suffices to show that if $\sigma,\tau\in\Omega_H$ then $\sigma\cap\tau\fleq\sigma$. So, let $U,U'\subseteq H$. We set $\sigma\dfgl(\bigcap_{L\in U}u_L^{\vee})\cap(\bigcap_{L\in H\setminus U}-u_L^{\vee})$ and $\tau\dfgl(\bigcap_{L\in U'}u_L^{\vee})\cap(\bigcap_{L\in H\setminus U'}-u_L^{\vee})$. For $L\in U'\setminus U$ we set $\overline{u}_L\dfgl-u_L$ and for $L\in U\setminus U'$ we set $\overline{u}_L\dfgl u_L$. Let $(L_i)_{i=1}^r$ be a counting of $\overline{U}\dfgl(U'\setminus U)\cup(U\setminus U')$. If $k\in[0,r-1]$ then $\sigma\cap(\bigcap_{i=1}^k\overline{u}_{L_i}^{\perp})\subseteq\overline{u}_{L_{k+1}}^{\vee}$ and $\sigma\cap(\bigcap_{i=1}^{k+1}\overline{u}_{L_i}^{\perp})=\sigma\cap(\bigcap_{i=1}^k\overline{u}_{L_i}^{\perp})\cap\overline{u}_{L_{k+1}}^{\perp}$, hence inductively we get $\sigma\cap\tau=\sigma\cap(\bigcap_{L\in\overline{U}}\overline{u}_L^{\vee})\fleq\sigma$ as claimed.

With the above notations it is clear that if $x\in V$ and $U\dfgl\{L\in H\mid x\in u_L^{\vee}\}$ then $$\textstyle x\in(\bigcap_{L\in U}u_L^{\vee})\cap(\bigcap_{L\in H\setminus U}-u_L^{\vee})\in\Omega_H,$$ showing that $\omegaq_H$ is complete. So, every cone in $\omegaq_H$ is the intersection of a family in $\Omega_H$ (\ref{top100}~b)), hence $s(\omegaq_H)=\bigcap\Omega_H=\bigcap_{L\in H}u_L^{\perp}$. Thus, $\omegaq_H$ is a $W$-fan if and only if $\bigcap_{L\in H}L^{\perp}=0$.
\end{no}

\begin{no}\label{comp20}
Let $\sig$ be a $W$-fan. We define a \textit{$W$-separating family for $\sig$} to be a family $H=(H_{\sigma,\tau})_{(\sigma,\tau)\in\ff(\sig)^2}$ of $W$-hyperplanes such that $H_{\sigma,\tau}$ separates $\sigma$ and $\tau$ in their intersection for all $\sigma,\tau\in\ff(\sig)$. Such a family exists (\ref{sep}), and if $\sigma\in\ff(\sig)$ then $H_{\sigma,\sigma}=\langle\sigma\rangle$. Moreover, the set $H'\dfgl\{H_{\sigma,\tau}^{\perp}\mid\sigma,\tau\in\ff(\sig)\}$ is a finite set of $W^*$-rational lines in $V^*$. The $W$-semifan $\omegaq_{H'}$ (\ref{comp10}) is denoted by $\omegaq_{\sig,H}$ and called \textit{the complete $W$-semifan associated with $\sig$ and $H$.}
\end{no}

\begin{lemma}\label{comp25}
Let $\sig$ be a complete $W$-semifan and let $X\subseteq V$ such that $\cl(\inn(X))=X$. If every cone in $\sig$ is contained in $X$ or in $\cl(V\setminus X)$, then $X=\bigcup\{\sigma\in\sig\mid\sigma\subseteq X\}$, $$\cl(V\setminus X)=\bigcup\{\sigma\in\sig\mid\sigma\subseteq\cl(V\setminus X)\},\,\text{ and }\,\,\fr(X)=\bigcup\{\sigma\in\sig\mid\sigma\subseteq\fr(X)\}.$$
\end{lemma}

\begin{proof}
Straightforward.
\end{proof}

\begin{proposition}\label{comp30}
Let $\sig$ be a noncomplete, equifulldimensional $W$-fan in $V$, let $H$ be a $W$-separating family for $\sig$, and let $\Tau\dfgl\{\sigma\in\omegaq_{\sig,H}\mid\sigma\subseteq\cl(V\setminus|\sig|)\}$ and $\Tau'\dfgl\{\sigma\in\omegaq_{\sig,H}\mid\sigma\subseteq\fr(|\sig|)\}$. Then, $\omegaq_{\sig,H}$ is a $W$-fan, $\Tau$ is a $W$-subfan of $\omegaq_{\sig,H}$ with $|\Tau|=\cl(V\setminus|\sig|)$ and $|\sig|\cup|\Tau|=V$, and $\Tau'$ is a $W$-subfan of $\Tau$ and a $W$-subdivision of $\fff(\sig)$.
\end{proposition}

\begin{proof}
As $\fff(\sig)$ is a nonempty $W$-fan it suffices to show that every cone in $\omegaq_{\sig,H}$ is contained in $|\sig|$ or in $\cl(V\setminus|\sig|)$, and that every element of $\Tau'$ is contained in a cone in $\fff(\sig)$ (\ref{comp25}, \ref{top70}). Let $\sigma\in\omegaq_{\sig,H}$ with $\dim(\sigma)=n$ and assume $\sigma\not\subseteq\cl(V\setminus|\sig|)$ and $\sigma\not\subseteq|\sig|$, so that $\sigma$ meets $\inn(|\sig|)$ and $V\setminus|\sig|$. There are $x\in\inn(\sigma)\cap\inn(|\sig|)$ and $y\in\inn(\sigma)\setminus|\sig|$ (\ref{topprop}~a)), hence $z\in\,\strecker x,y\streckel\,\cap\fr(|\sig|)$ with $z\in\inn(\sigma)$ and $\tau\in\ff(\sig)$ with $z\in\tau$ (\ref{top70}). But $\sigma$ lies on one side of $\langle\tau\rangle$ by construction of $\omegaq_{\sig,H}$, hence $\inn(\sigma)$ and in particular $z$ lie strictly on one side of $\langle\tau\rangle$. This contradicts $z\in\tau$, and therefore $\sigma\subseteq|\sig|$ or $\sigma\subseteq\cl(V\setminus|\sig|)$ for every $\sigma\in\omegaq_{\sig,H}$.

Next, let $\tau\in\Tau'$ and assume $\tau\not\subseteq\sigma$ for every $\sigma\in\ff(\sig)$, so that there are $\sigma,\sigma'\in\ff(\sig)$ and $x,y\in\tau$ with $x\in\sigma\setminus\sigma'$ and $y\in\sigma'\setminus\sigma$. Then, $H_{\sigma,\sigma'}$ separates $x$ and $y$ strictly, contradicting that $\tau$ lies on one side of $H_{\sigma,\sigma'}$. Therefore, $\tau$ is contained in a cone in $\fff(\sig)$, and thus the claim is proven.
\end{proof}


\subsection{Adjusting extensions}\mbox{}\medskip

\textit{Throughout this subsection let $\sig\subseteq\sig'$ be a separable $W$-extension of fans, let $\Tau\subseteq\sig'$ be a $W$-subfan, and let $\Tau'$ be a $W$-subdivision of $\Tau$.}\medskip

We will later (in a more special situation) face the problem of ``adjusting'' $\sig'$ such that $\Tau$ is turned into $\Tau'$ but $\sig$ remains unchanged. To allow a solution to this, $\Tau$ has to be ``independent of $\sig$'' in some way, and hence it is not astonishing that separability is a key property in the construction of adjustments.\medskip

\begin{no}\label{adj10}
If $\sigma\in\sig'$ and $\tau\in\Tau'$ then the sum $\inn_{\sig}(\sigma)+(\exte_{\sig}(\sigma)\cap\tau)$ is direct. Thus, $$\adj_{\sig}(\sig',\Tau')\dfgl\{\inn_{\sig}(\sigma)\oplus(\exte_{\sig}(\sigma)\cap\tau)\mid\sigma\in\sig'\wedge\tau\in\Tau'\}$$ is a finite set of sharp $W$-polycones, called \textit{the adjustment of $\sig'$ to $\Tau'$ over $\sig$.} If $\sigma\in\sig'$ and $\tau\in\Tau'$ then $\sigma\cap\tau\fleq\tau$, hence $\adj_{\sig}(\sig',\Tau')=\{\inn_{\sig}(\sigma)\oplus\tau\mid\sigma\in\sig'\wedge\tau\in\Tau'\wedge\tau\subseteq\exte_{\sig}(\sigma)\}$.
\end{no}

\begin{lemma}\label{inexadj}
{\rm a)} If $\sigma\in\sig'$ and $\tau\in\Tau'$ then $$\face(\inn_{\sig}(\sigma)\oplus(\exte_{\sig}(\sigma)\cap\tau))=\{\inn_{\sig}(\sigma')\oplus(\exte_{\sig}(\sigma')\cap\tau')\mid\sigma'\fleq\sigma\wedge\tau'\fleq\tau\}.$$

{\rm b)} If $\sigma,\sigma'\in\sig'$, $\tau,\tau'\in\Tau'$, $\eta\dfgl\sigma\cap\sigma'$ and $\zeta\dfgl\tau\cap\tau'$ then $$(\inn_{\sig}(\sigma)\oplus(\exte_{\sig}(\sigma)\cap\tau))\cap(\inn_{\sig}(\sigma')\oplus(\exte_{\sig}(\sigma')\cap\tau'))=\inn_{\sig}(\eta)\oplus(\exte_{\sig}(\eta)\cap\zeta).$$

{\rm c)} $\adj_{\sig}(\sig',\Tau')$ is a $W$-fan.
\end{lemma}

\begin{proof}
a) follows from \ref{dec20} and \ref{pack20}. b) Let $\omega\dfgl\inn_{\sig}(\sigma)\oplus(\exte_{\sig}(\sigma)\cap\tau)$, $\omega'\dfgl\inn_{\sig}(\sigma')\oplus(\exte_{\sig}(\sigma')\cap\tau')$ and $\thet\dfgl\inn_{\sig}(\eta)\oplus(\exte_{\sig}(\eta)\cap\zeta)$. Keeping in mind \ref{pack20} we get $$\thet=(\inn_{\sig}(\sigma)\cap\inn_{\sig}(\sigma'))\oplus(\exte_{\sig}(\sigma)\cap\exte_{\sig}(\sigma')\cap\zeta)\subseteq\omega\cap\omega'\subseteq\eta=\inn_{\sig}(\eta)\oplus\exte_{\sig}(\eta).$$ Conversely, let $x\in\omega\cap\omega'$. By the above there are $y\in\inn_{\sig}(\eta)$, $y'\in\exte_{\sig}(\eta)$, $z\in\inn_{\sig}(\sigma)$, $z'\in\exte_{\sig}(\sigma)\cap\tau$, $w\in\inn_{\sig}(\sigma')$ and $w'\in\exte_{\sig}(\sigma')\cap\tau'$ with $x=y+y'=z+z'=w+w'$. It follows $y-z=z'-y'\in\langle\inn_{\sig}(\sigma)\rangle\cap\langle\exte_{\sig}(\sigma)\rangle=0$ and $y-w=w'-y'\in\langle\inn_{\sig}(\sigma')\rangle\cap\langle\exte_{\sig}(\sigma')\rangle=0$, hence $y=z=w\in\inn_{\sig}(\eta)$ and $y'=z'=w'\in\exte_{\sig}(\eta)\cap\zeta$, and therefore $x\in\thet$. This proves b), and c) follows immediately.
\end{proof}

\begin{proposition}\label{adjthm1}
Suppose that $\exte_{\sig}(\sig')=\Tau$.

{\rm a)} $\adj_{\sig}(\sig',\Tau')$ is a $W$-subdivision of $\sig'$, and the $W$-subdivisions of the $W$-subfans $\sig\subseteq\sig'$ and $\Tau\subseteq\sig'$ induced by $\adj_{\sig}(\sig',\Tau)$ are $\sig$ and $\Tau'$, respectively.

{\rm b)} $\adj_{\sig}(\sig',\Tau')$ is a separable $W$-extension of $\sig$, and if $\sigma\in\sig'$ and $\tau\in\Tau'$ then $$\inn_{\sig}(\inn_{\sig}(\sigma)\oplus(\exte_{\sig}(\sigma)\cap\tau))=\inn_{\sig}(\sigma)\,\text{ and }\,\exte_{\sig}(\inn_{\sig}(\sigma)\oplus(\exte_{\sig}(\sigma)\cap\tau))=\exte_{\sig}(\sigma)\cap\tau.$$
\end{proposition}

\begin{proof}
Easy on use of \ref{inexadj}.
\end{proof}

\begin{proposition}\label{adjthm2}
Suppose that $\exte_{\sig}(\sig')=\Tau$.

{\rm a)} $\adj_{\sig}(\sig',\Tau')$ is relatively simplicial over $\sig$ if and only if $\Tau'$ is simplicial.

{\rm b)} $\adj_{\sig}(\sig',\Tau')$ is tightly separable over $\sig$ if and only if $\sig'$ is so.

{\rm c)} $\adj_{\sig}(\sig',\Tau')$ is a $W$-quasipacking of $\sig$ if and only if $\sig'$ is so.

{\rm d)} $\adj_{\sig}(\sig',\Tau')$ is a $W$-packing of $\sig$ if and only if $\sig'$ is so and $\Tau'$ is simplicial.
\end{proposition}

\begin{proof}
Easy on use of \ref{adjthm1} and \ref{top70}.
\end{proof}

\begin{proposition}\label{adjthm3}
Suppose that $\exte_{\sig}(\sig')=\Tau$, and let $\Omega$ be a $W$-extension of $\Tau'$ with $|\sig'|\cap|\Omega|=|\Tau|$. Then, $\sighut\dfgl\adj_{\sig}(\sig',\Tau')\cup\Omega$ is a $W$-extension of $\adj_{\sig}(\sig',\Tau')$ with $|\sighut|=|\sig'|\cup|\Omega|$.
\end{proposition}

\begin{proof}
By \ref{adjthm1} it suffices to show that the intersection of a cone in $\adj_{\sig}(\sig',\Tau')$ and a cone in $\Omega$ is a common face of both. So, let $\sigma\in\sig'$, let $\tau\in\Tau'$, let $\omega\in\Omega$, and let $x\in(\inn_{\sig}(\sigma)\oplus(\exte_{\sig}(\sigma)\cap\tau))\cap\omega$. Then, we have $x\in|\adj_{\sig}(\sig',\Tau')|\cap|\Omega|=|\Tau'|$, hence there are $\rho\in\Tau'$ and $\thet\in\Tau$ with $x\in\rho\subseteq\thet$. As $\thet=\exte_{\sig}(\thet)$ it follows $\rho=\inn_{\sig}(\thet)\oplus(\exte_{\sig}(\thet)\cap\tau)\in\adj_{\sig}(\sig',\Tau')$, and as $\inn_{\sig}(\thet)=0$ we get $$x\in(\inn_{\sig}(\sigma)\oplus(\exte_{\sig}(\sigma)\cap\tau))\cap\rho=\exte_{\sig}(\sigma\cap\thet)\cap(\tau\cap\rho)\subseteq\tau$$ (\ref{inexadj}). This implies $(\inn_{\sig}(\sigma)\oplus(\exte_{\sig}(\sigma)\cap\tau))\cap\omega=\tau\cap\omega$, thus the claim (\ref{dec25}, \ref{adj10}).
\end{proof}


\subsection{Pulling back extensions}\mbox{}\medskip

\textit{Throughout this subsection let $\sig$ be a $W$-fan, and let $\xi\in\sig_1$. We set $\Tau\dfgl\sig/\xi$, $p\dfgl p_{\xi}:V\twoheadrightarrow V_{\xi}$ and $\hap\dfgl\hap_{\xi}:\sig_{\xi}\rightarrow\Tau$ (see \ref{proj}), and we denote by $\Lambda$ the set of $1$-dimensional sharp $W_{\xi}$-polycones.}\medskip

Our plan for constructing a packing of $\sig$ is to ``pack'' recursively each of its $1$-di\-men\-sion\-al cones. This will be achieved inductively on the dimension of $V$ by projecting along the cones to ``pack''. Hence we need a technique for pulling back an extension along such a projection. Our construction depends on the choice of a ``pullback datum''. Existence of ``pullback data'', proved in \ref{pullex}, is the only point where we choose and use a (Hilbert) norm on $V$.\medskip

\begin{no}
A \textit{$W$-pullback datum along $\xi$} is a triple $(q,a,B)$ such that $q$ is a $(W_{\xi},W)$-rational section of $p$, that $a\in W\cap\xi\setminus 0$, and that $B=(b_{\rho})_{\rho\in\Lambda}\in\prod_{\rho\in\Lambda}(W_{\xi}\cap\rho\setminus 0)$.

Let $(q,a,B)$ be a $W$-pullback datum along $\xi$, and let $\sigma$ be a $W_{\xi}$-polycone that is separable over $\Tau$ such that $\exte_{\Tau}(\sigma)$ is simplicial. Then, $B_{\sigma}\dfgl\{b_{\rho}\mid\rho\in\exte_{\Tau}(\sigma)_1\}$ is free, and therefore the $W$-polycone $\cone(q(B_{\sigma})+a)$ is simplicial. It is readily checked that the sum of $W$-polycones $\xi+\cone(q(B_{\sigma})+a)$ is direct, and from this it follows that the sum of $W$-polycones $\hap^{-1}(\inn_{\Tau}(\sigma))+\cone(q(B_{\sigma})+a)$ is direct, too. Thus, $$\psi_{q,a,B}(\sigma)\dfgl\hap^{-1}(\inn_{\Tau}(\sigma))\oplus\cone(q(B_{\sigma})+a)$$ is a sharp $W$-polycone (\ref{dec25}). If no confusion can arise we denote it by $\psi(\sigma)$.
\end{no}

\begin{lemma}\label{pull12}
Let $(q,a,B)$ be a $W$-pullback datum along $\xi$, and let $\sigma,\tau$ be $W_{\xi}$-polycones that are separable over $\Tau$ such that $\exte_{\Tau}(\sigma)$ and $\exte_{\Tau}(\tau)$ are simplicial.

{\rm a)} If $\sigma\cap\sigma'\in\face(\sigma)\cap\face(\sigma')$ then $\psi_{q,a,B}(\sigma)\cap\psi_{q,a,B}(\sigma')=\psi_{q,a,B}(\sigma\cap\sigma')$.

{\rm b)} If $\tau\fleq\sigma$ then $\psi_{q,a,B}(\tau)=\psi_{q,a,B}(\sigma)\cap(\psi_{q,a,B}(\tau)-\xi)\fleq\psi_{q,a,B}(\sigma)$.
\end{lemma}

\begin{proof}
a) The inclusion ``$\supseteq$'' follows from \ref{pack20}. As $p(\psi(\sigma)\cap\psi(\sigma'))\subseteq p(\psi(\sigma\cap\sigma'))$ we have $$\psi(\sigma)\cap\psi(\sigma')\subseteq\psi(\sigma)\cap(\psi(\sigma\cap\sigma')-\xi),$$ hence the remaining inclusion will follow from b).

b) The inclusion ``$\subseteq$'' follows from the inclusion ``$\supseteq$'' in a). Conversely, let $y\in\psi(\tau)$ and $z\in\xi$ with $x\dfgl y-z\in\psi(\sigma)$. There are unique $x_0\in\hap^{-1}(\inn_{\Tau}(\sigma))$ and $x_1\in\cone(q(B_{\sigma})+a)$ with $x=x_0+x_1$, and $y_0\in\hap^{-1}(\inn_{\Tau}(\tau))\fleq\hap^{-1}(\inn_{\Tau}(\sigma))$ and $y_1\in\cone(q(B_{\tau})+a)\subseteq\cone(q(B_{\sigma})+a)$ with $y=y_0+y_1$ (\ref{dec10}, \ref{pack20}). As $\xi\fleq\hap^{-1}(\inn_{\Tau}(\sigma))$ we have $x_0+z\in\hap^{-1}(\inn_{\Tau}(\sigma))$, hence we get $x_0+z=y_0\in\hap^{-1}(\inn_{\Tau}(\tau))$ and $x_1=y_1\in\cone(q(B_{\tau})+a)$. Moreover, there is a $u\in V^*$ with $\hap^{-1}(\inn_{\Tau}(\tau))=\hap^{-1}(\inn_{\Tau}(\sigma))\cap u^{\perp}$ (\ref{pack20}), and as $\xi\fleq\hap^{-1}(\inn_{\Tau}(\tau))$ we have $u(z)=0$. This yields $u(x_0)=u(x_0+z)=u(y_0)=0$, hence $x_0\in\hap^{-1}(\inn_{\Tau}(\tau))$ and therefore $x\in\psi(\tau)$.

Finally, if $v\in V^*_{\xi}$ with $\sigma\subseteq v^{\vee,V_{\xi}}$ and $\tau=\sigma\cap v^{\perp,V_{\xi}}$, then setting $w\dfgl v\circ p\in V^*$ we get $\psi(\sigma)\subseteq w^{\vee,V}$ and $\psi(\sigma)\cap w^{\perp,V}=\psi(\sigma)\cap(\psi(\tau)-\xi)$.
\end{proof}

\begin{no}
Let $\Tau\subseteq\Tau'$ be a relatively simplicial, separable $W_{\xi}$-extension and let $(q,a,B)$ be a $W$-pullback datum along $\xi$. We set $\Psi_{q,a,B}(\Tau')\dfgl\{\psi_{q,a,B}(\sigma)\mid\sigma\in\Tau'\}$ and $\psiquer_{q,a,B}(\Tau')\dfgl\bigcup_{\sigma\in\Psi_{q,a,B}(\Tau')}\face(\sigma)$. If no confusion can arise we denote these sets by $\Psi(\Tau')$ and $\psiquer(\Tau')$. The set $\sig_{q,a,B}(\Tau')\dfgl\sig\cup\psiquer_{q,a,B}(\Tau')$ is called \textit{the pullback of\/ $T'$ along $\xi$ by means of $(q,a,B)$ over $\sig$,} and if no confusion can arise we denote it by $\sig(\Tau')$.

A $W$-pullback datum $(q,a,B)$ along $\xi$ is called \textit{good for $\sig$ and $\Tau'$} if for every $\sigma\in\sig\setminus\psiquer_{q,a,B}(\Tau')$ and every $\tau\in\Tau'$ we have $\sigma\cap\psi_{q,a,B}(\tau)=\sigma\cap\psi_{q,a,B}(\inn_{\Tau}(\tau))$, and it is called \textit{very good for $\sig$ and $\Tau'$} if it is good for $\sig$ and $\Tau'$ and $\cone(q(B_{\tau})+a)$ is free over $\sig$ for every $\tau\in\Tau'$.
\end{no}

\begin{proposition}\label{pulla}
Let $\Tau\subseteq\Tau'$ be a relatively simplicial, separable $W_{\xi}$-extension, let $\sig'\subseteq\sig$ be a $W$-subfan, and let $(q,a,B)$ be a $W$-pullback datum along $\xi$ that is good for $\sig$ and $\Tau'$.

{\rm a)} $\sig'\subseteq\sig_{q,a,B}(\Tau')$ is a $W$-extension with $\sig_{q,a,B}(\Tau')/\xi=\Tau'$.

{\rm b)} If $\sig'\subseteq\sig$ is relatively simplicial then $\sig'\subseteq\sig_{q,a,B}(\Tau')$ is relatively simplicial.

{\rm c)} If $\sig'\subseteq\sig$ is separable and $(q,a,B)$ is very good for $\sig$ and $\Tau'$, then $\sig'\subseteq\sig_{q,a,B}(\Tau')$ is separable.
\end{proposition}

\begin{proof}
a) By \ref{pull12} it suffices to show that if $\sigma\in\sig\setminus\psiquer(\Tau')$ and $\tau\in\Tau'$ then $$\sigma\cap\psi(\tau)\in\face(\sigma)\cap\face(\psi(\tau)).$$ This follows easily on use of \ref{dec25}.

b) If $\sigma\in\sig(\Tau')\setminus\sig$ is free over $\sig'$ then there are $\tau\in\Tau'$, $\omega\fleq\hap^{-1}(\inn_{\Tau}(\tau))$ and $\omega'\fleq\cone(q(B_{\tau})+a)$ with $\sigma=\omega\oplus\omega'$ (\ref{dec20}). As $\sigma$ is free over $\sig'$ the same holds for $\omega$, hence $\omega$ is simplicial, and as $\cone(q(B_{\tau})+a)$ is simplicial the same holds for $\omega'$. Thus, $\sigma$ is simplicial (\ref{dec70}).

c) follows immediately from \ref{pack20}.
\end{proof}

\begin{lemma}\label{pulll}
Let $\Tau\subseteq\Tau'$ be a relatively simplicial, separable $W_{\xi}$-extension, and let $(q,a,B)$ be a $W$-pullback datum along $\xi$ that is good for $\sig$ and $\Tau'$.

{\rm a)} $\ff(\sig(\Tau'))=\{\sigma\in\ff(\sig)\mid\psiquer(\Tau')_{\sigma}\subseteq\sig_{\sigma}\}\cup\{\sigma\in\ff(\psiquer(\Tau'))\mid\sig_{\sigma}\subseteq\psiquer(\Tau')_{\sigma}\}$.

{\rm b)} $\Tau'$ is equifulldimensional if and only if $\psiquer(\Tau')$ is so, and then $\dd(\sig(\Tau'))=\dd(\sig)\setminus\psiquer(\Tau')$.
\end{lemma}

\begin{proof}
a) is straightforward to prove, and b) follows readily on use of \ref{pull12}.
\end{proof}

\begin{lemma}\label{tslem}
Let $\Tau\subseteq\Tau'$ be an equifulldimensional, relatively simplicial, separable $W_{\xi}$-extension, let $(q,a,B)$ be a $W$-pullback datum along $\xi$ that is very good for $\sig$ and $\Tau'$, let $\sigma\in\Tau'\setminus\Tau$, and let $\tau\in\sig_{q,a,B}(\Tau')$. If\/ $\xi\not\fleq\tau$ and $\cone(q(B_{\sigma})+a)\fleq\tau\fleq\psi_{q,a,B}(\sigma)$, then $\tau\subseteq\fr(|\sig_{q,a,B}(\Tau')|)$.
\end{lemma}

\begin{proof}
We have $\sig(\Tau')_{\tau}\subseteq\psiquer(\Tau')_{\tau}$, for if $\omega\in\sig(\Tau')_{\tau}\setminus\psiquer(\Tau')$ then $\omega\in\sig\setminus\sig_{\xi}$, hence $$\cone(q(B_{\sigma})+a)\fleq\tau\fleq\omega\cap\psi(\sigma)=\omega\cap\psi(\inn_{\Tau}(\sigma))\in\sig$$ and in particular $\cone(q(B_{\sigma})+a)=0$ by hypothesis on $(q,a,B)$, yielding the contradiction $\tau\in\Tau$. Furthermore, we have $\psiquer(\Tau')_{\tau}\cap\psiquer(\Tau')_n\neq\emptyset$ (\ref{pulll}~b)) and $\psiquer(\Tau')_n\subseteq\psiquer(\Tau')_{\xi}$. If for every $$\rho\in\psiquer(\Tau')_{\tau}\cap\psiquer(\Tau')_{n-1}$$ there are $\omega,\omega'\in\psiquer(\Tau')_{\rho}\cap\psiquer(\Tau')_n$ with $\rho=\omega\cap\omega'$, then it follows $\xi\fleq\rho$, yielding the contradiction $$\textstyle\xi\fleq\bigcap_{\omega\in\psiquer(\Tau')_{\tau}\cap\psiquer(\Tau')_n}\bigcap\{\rho\in\psiquer(\Tau')_{n-1}\mid\tau\fleq\rho\fleq\omega\}=\tau$$ (\ref{topprop}~b)). So, there is a $\rho\in\ff(\sig(\Tau'))$ with $\tau\fleq\rho$, implying $\sig_{\rho}\subseteq\sig(\Tau')_{\tau}\cap\sig(\Tau')_{\rho}\subseteq\psiquer(\Tau')_{\tau}\cap\sig(\Tau')_{\rho}=\psiquer(\Tau')_{\rho}$, hence $\rho\in\ff(\sig(\Tau'))$ (\ref{pulll}~a)) and therefore $\tau\subseteq\rho\subseteq\fr(|\sig(\Tau')|)$ (\ref{top80}).
\end{proof}

\begin{proposition}\label{pullb}
Let $\Tau\subseteq\Tau'$ be a relatively simplicial, separable $W_{\xi}$-extension, let $\sig'\subseteq\sig$ be a $W$-subfan, and let $(q,a,B)$ be a $W$-pullback datum along $\xi$ that is very good for $\sig$ and $\Tau'$. If $\sig'\subseteq\sig$ is tightly separable and $\Tau'$ is equifulldimensional, then $\sig'\subseteq\sig_{q,a,B}(\Tau')$ is tightly separable.
\end{proposition}

\begin{proof}
Let $\sigma\in\sig(\Tau')\setminus\sig'$. As $\sig(\Tau')$ is separable over $\sig'$ (\ref{pulla}~c)) it suffices to show $\exte_{\sig'}(\sigma)\setminus 0\subseteq\fr(|\sig(\Tau')|)$. If $\sigma\notin\sig$ then there is a $\tau\in\Tau'\setminus\Tau$ with $\sigma\fleq\psi(\tau)$, hence $\exte_{\sig'}(\sigma)\fleq\psi(\tau)$, and the claim follows from \ref{tslem}. Suppose $\sigma\in\sig$. If $\exte_{\sig'}(\sigma)\fleq\omega$ for an $\omega\in\ff(\sig)$ with $\psiquer(\Tau')_{\omega}\subseteq\sig_{\omega}$ or an $\omega\in\dd(\sig)\setminus\psiquer(\Tau')$, then the claim follows from \ref{pulll}~a) and \ref{top80}. So, suppose $\exte_{\sig'}(\sigma)\not\fleq\omega$ for all $\omega\in\ff(\sig)$ with $\psiquer(\Tau')_{\omega}\subseteq\sig_{\omega}$ and all $\omega\in\dd(\sig)\setminus\psiquer(\Tau')$. We first consider the case that $\exte_{\sig'}(\sigma)\fleq\omega$ for an $\omega\in\ff(\sig)$. As $\psiquer(\Tau')_{\omega}\not\subseteq\sig_{\omega}$ there is a $\tau\in\Tau'\setminus\Tau$ with $\omega\fleq\psi(\tau)$, hence $\exte_{\sig'}(\sigma)\fleq\psi(\tau)$, and the claim follows from \ref{tslem}. Finally, we consider the case that $\exte_{\sig'}(\sigma)\not\fleq\omega$ for all $\omega\in\ff(\sig)$. As $\exte_{\sig'}(\sigma)\subseteq\fr(|\sig|)$ there is an $\omega\in\dd(\sig)\cap\psiquer(\Tau')$ with $\exte_{\sig'}(\sigma)\fleq\omega$. Equifulldimensionality of $\Tau'$ implies that there is a $\tau\in\Tau'\setminus\Tau$ with $\omega\fleq\psi(\tau)$, hence $\exte_{\sig'}(\sigma)\fleq\psi(\tau)$, and the claim follows from \ref{tslem}.
\end{proof}

\begin{proposition}\label{pullc}
Let $\Tau\subseteq\Tau'$ be a relatively simplicial, separable $W_{\xi}$-extension, let $\sig'\subseteq\sig$ be a $W$-subfan, and let $(q,a,B)$ be a $W$-pullback datum along $\xi$ that is good for $\sig$ and $\Tau'$.

{\rm a)} If $\Tau'$ is complete then $C_{\sig'}(\sig_{q,a,B}(\Tau'))=C_{\sig'}(\sig)\setminus\{\xi\}$.

{\rm b)} If $\xi\in\sig'$ and $\sig_{\max}\subseteq\bigcup_{\rho\in\sig'_1}\sig_{\rho}$, then $\sig_{q,a,B}(\Tau')_{\max}\subseteq\bigcup\{\sig_{q,a,B}(\Tau')_{\rho}\mid\rho\in\sig'_1\setminus C_{\sig'}(\sig_{q,a,B}(\Tau'))\}$.

{\rm c)} If $\Tau'$ is complete and maximal elements of $\bigcup_{\rho\in\sig'_1\setminus C_{\sig'}(\sig)}\sig_{\rho}$ are full, then maximal elements of $\bigcup_{\rho\in\sig'_1\setminus C_{\sig'}(\sig_{q,a,B}(\Tau'))}\sig_{q,a,B}(\Tau')_{\rho}$ are full.
\end{proposition}

\begin{proof}
a) Let $\rho\in C_{\sig'}(\sig)\setminus\{\xi\}$. There is a $\sigma\in\dd(\sig)\cup\ff(\sig)$ with $\rho\fleq\sigma$, and we have to show that there is an $\omega\in\dd(\sig(\Tau'))\cup\ff(\sig(\Tau'))$ with $\rho\fleq\omega$ (\ref{qpc}, \ref{top80}). 

If $\sigma\notin\psiquer(\Tau')$ then $\omega\dfgl\sigma$ fulfils the claim (\ref{pulll}~a)). Suppose $\sigma\in\psiquer(\Tau')$, hence $\sigma\in\ff(\sig)$ (\ref{pulll}~b)). Let $\Xi\dfgl\{\tau\in\Tau'_{n-1}\mid\sigma\fleq\psi(\tau)\}$, and for $\tau\in\Xi$ let $\Xi_{\tau}\dfgl\{\thet\in\psi(\tau)_{n-1}\mid\rho\fleq\thet\}$ so that $\rho=\bigcap\Xi_{\tau}$ (\ref{topprop}~b)). If there are $\tau\in\Xi$ and $\thet\in\Xi_{\tau}\cap\ff(\sig(\Tau'))$ then $\omega\dfgl\thet$ fulfils the claim. Suppose $\Xi_{\tau}\cap\ff(\sig(\Tau'))=\emptyset$ for every $\tau\in\Xi$.

Now, we assume that there is a $\tau\in\Xi\setminus\Tau$, hence for every $\thet\in\Xi_{\tau}$ there is an $\eta^{(\thet)}\in\sig(\Tau')_n$ with $\thet=\eta^{(\thet)}\cap\psi(\tau)$. If $\eta^{(\thet)}\in\psiquer(\Tau')$ and hence $\xi\fleq\eta^{(\thet)}$ for every $\thet\in\Xi_{\tau}$, then we get the contradiction $\xi\fleq\bigcap\Xi_{\tau}=\rho$. So, there is a $\thet\in\Xi_{\tau}$ with $\eta^{(\thet)}\in\sig\setminus\psiquer(\Tau')$, hence $\xi\not\fleq\eta^{(\thet)}\cap\psi(\tau)=\thet$. As $\tau\notin\Tau$ and as $(q,a,B)$ is good for $\sig$ and $\Tau$ we have $\thet=\eta^{(\thet)}\cap\psi(\inn_{\Tau}(\tau))\fleq\psi(\inn_{\Tau}(\tau))\flneq\psi(\tau)$, but then $\dim(\thet)=\dim(\psi(\tau))-1$ yields the contradiction $\xi\fleq\psi(\inn_{\Tau}(\tau))=\thet$. This shows $\Xi\subseteq\Tau$, hence $\psiquer(\Tau')_{\sigma}\subseteq\sig_{\sigma}$, and therefore $\omega\dfgl\sigma$ fulfils the claim (\ref{pulll}~a)).

Conversely, let $\rho\in C_{\sig'}(\sig(\Tau'))$. As $\sig/\rho$ is a $W$-subfan of the noncomplete $W$-fan $\sig(\Tau')/\rho$ it follows $\rho\in C_{\sig'}(\sig)$, and completeness of $\sig(\Tau')/\xi=\Tau'$ (\ref{pulla}~a)) yields $\rho\neq\xi$.

b) As $\psiquer(\Tau')_{\max}\subseteq\sig(\Tau')_{\xi}$ the hypotheses imply $$\textstyle\sig(\Tau')_{\max}\subseteq\sig_{\max}\cup\psiquer(\Tau')_{\max}\subseteq(\bigcup_{\rho\in\sig'_1}\sig_{\rho})\cup\sig(\Tau')_{\xi}\subseteq\bigcup_{\rho\in\sig'_1}\sig(\Tau')_{\rho}.$$

c) follows readily on use of a) and \ref{pulll}~b).
\end{proof}

\begin{lemma}\label{pullexlemma}
Let $\norm$ be a norm on $V$, let $a\in\xi\setminus 0$, and let $\sigma,\tau$ be $W$-polycones with $\sigma\cap\tau\in\face(\sigma)\cap\face(\tau)$ and $\xi\in\face(\tau)\setminus\face(\sigma)$. There is an $\eps_0\in\R_{>0}$ such that for every $\eps\in]0,\eps_0[$ and every finite subset $B\subseteq V$ with $\nrm{b}\leq\eps$ for every $b\in B$ we have $\sigma\cap(\tau+\cone(B+a))=\sigma\cap\tau$.
\end{lemma}

\begin{proof}
As $\sigma\cap\tau$ is not full there is a $u\in V^*\setminus 0$ with $\sigma\subseteq u^{\vee}$ and $\tau\subseteq(-u)^{\vee}$ such that $\sigma\cap u^{\perp}=\sigma\cap\tau=\tau\cap u^{\perp}$ (\ref{sep}). If $a\in u^{\vee}$ we get $a\in\xi\cap u^{\vee}\subseteq\tau\cap u^{\perp}\subseteq\sigma$, hence the contradiction $\xi\fleq\sigma$. Therefore, denoting by $d$ the distance on $V$ induced by $\norm$ we have $\eps_0\dfgl d(a,u^{\vee})\in\R_{>0}$. Let $\eps\in]0,\eps_0[$ and let $B\subseteq V$ be finite with $\nrm{b}\leq\eps$ for every $b\in B$. Then, $B+a\subseteq(-u)^{\vee}$, for otherwise there is a $b\in B$ with $b+a\in u^{\vee}$, yielding the contradiction $\eps\geq\nrm{b}=d(a,b+a)\geq\eps_0>\eps$. It follows $\tau+\cone(B+a)\subseteq(-u)^{\vee}$, hence $\sigma\cap\tau\subseteq\sigma\cap(\tau+\cone(B+a))\subseteq\sigma\cap u^{\perp}=\sigma\cap\tau$ and thus the claim.
\end{proof}

\begin{proposition}\label{pullex}
There exists a $W$-pullback datum along $\xi$ that is very good for $\sig$ and every relatively simplicial, separable $W_{\xi}$-extension of $\Tau$.
\end{proposition}

\begin{proof}
We can choose a $W$-rational Hilbert norm $\norm$ on $V$ which induces a $W_{\xi}$-rational Hilbert norm $\norm'$ on $V_{\xi}$ and defines a $(W_{\xi},W)$-rational section $q:V_{\xi}\rightarrowtail V$ of $p$, inducing by coastriction the canonical isomorphism of $\R$-Hilbert spaces from $V_{\xi}$ onto the orthogonal complement of $\langle\xi\rangle$ in $V$ with respect to $\norm$ (\ref{hilb}). Let $d$ denote the distance on $V$ induced by $\norm$. We can choose an $a\in W\cap\xi\setminus 0$. If $\sigma\in\sig\setminus\sig_{\xi}$ then $d(a,\sigma)>0$, hence there is an $\eps_0\in\R_{>0}$ such that if $\sigma\in\sig\setminus\sig_{\xi}$ then $\eps_0\leq d(a,\sigma)$. For $\sigma\in\sig\setminus\sig_{\xi}$ and $\tau\in\sig_{\xi}$ there is an $\eps_{\sigma,\tau}\in\R_{>0}$ such that for every $\eps\in]0,\eps_{\sigma,\tau}[$ and every finite subset $B\subseteq W_{\xi}$ with $\nrm{b}'\leq\eps$ for every $b\in B$ we have $\sigma\cap(\tau+\cone(q(B)+a))=\sigma\cap\tau$ (\ref{pullexlemma}). Moreover, there is an $\eps_1\in\R_{>0}$ such that if $\sigma\in\sig\setminus\sig_{\xi}$ and $\tau\in\sig_{\xi}$ then $\eps_1\leq\eps_{\sigma,\tau}$. For every $\rho\in\Lambda$ there is a $b_{\rho}\in W_{\xi}\cap\rho\setminus 0$ with $\nrm{b_{\rho}}'<\min\{\eps_0,\eps_1\}$, for $W_{\xi}$ is dense in $V_{\xi}$. Setting $B\dfgl(b_{\rho})_{\rho\in\Lambda}$ it is clear that $(q,a,B)$ is a $W$-pullback datum along $\xi$.

Let $\Tau\subseteq\Tau'$ be a relatively simplicial, separable $W_{\xi}$-extension. Let $\tau\in\Tau'$ and let $\sigma\in\sig\setminus\psiquer(\Tau')\subseteq\sig\setminus\sig_{\xi}$. As $\psi(\inn_{\Tau}(\tau))\in\sig_{\xi}$ and as $B_{\tau}\subseteq W_{\xi}$ is a finite subset with $\nrm{b}'<\eps_1\leq\eps_{\sigma,\psi(\inn_{\Tau}(\tau))}$ for every $b\in B_{\tau}$ we get $\sigma\cap\psi(\tau)=\sigma\cap\psi(\inn_{\Tau}(\tau))$. Therefore, $(q,a,B)$ is good for $\sig$ and $\Tau'$.

Next, let $\tau\in\Tau'$. We will show that $\omega\dfgl\cone(q(B_{\tau})+a)$ is free over $\sig$. If $\sigma\in\sig_{\xi}$ then since $\exte_{\Tau}(\tau)$ is free over $\Tau$ we get $$\omega\cap\sigma=\omega\cap\psi(\exte_{\Tau}(\tau))\cap\psi(p(\sigma))=\omega\cap\psi(\exte_{\Tau}(\tau)\cap p(\sigma))=\omega\cap\xi=0$$ (\ref{pull12}~a)). So, let $\sigma\in\sig\setminus\sig_{\xi}$ and assume that there is an $x\in\omega\cap\sigma\setminus 0$. There is an $r\in\R_{>0}$ with $rx\in\conv(q(B_{\tau})+a)\cap\sigma$, hence a family $(r_{\rho})_{\rho\in\exte_{\Tau}(\tau)_1}$ in $\R_{\geq 0}$ with $\sum_{\rho\in\exte_{\Tau}(\tau)_1}r_{\rho}=1$ and $rx=\sum_{\rho\in\exte_{\Tau}(\tau)_1}r_{\rho}(q(b_{\rho})+a)$. This yields the contradiction $$\textstyle d(a,rx)=\nrm{rx-a}=\nrm{q(\sum_{\rho\in\exte_{\Tau}(\tau)_1}r_{\rho}b_{\rho})}=$$$$\textstyle\nrm{\sum_{\rho\in\exte_{\Tau}(\tau)_1}r_{\rho}b_{\rho}}'\leq\sum_{\rho\in\exte_{\Tau}(\tau)_1}r_{\rho}\nrm{b_{\rho}}'<\eps_0\leq d(a,\sigma)\leq d(a,rx).$$ So, $\omega\cap\sigma=0$, and thus $(q,a,B)$ is very good for $\sig$ and $\Tau'$.
\end{proof}


\section{Existence of completions}\label{sec9}

Finally we put everything together and prove our main result.\medskip

\begin{lemma}\label{8.10}
Let $n>1$, let $\sig$ be a $W$-fan with $\sig_1\neq\emptyset$, and suppose that for every $\R$-vector space $V'$ with $\dim(V')<n$, every $K$-structure $W'$ on $V'$ and every $W'$-fan $\sig'$ there exists a strong $W'$-completion of $\sig'$. There exists an increasing sequence $(\sig^{(i)})_{i\in\Nn}$ of relatively simplicial, tightly separable $W$-extensions of $\sig$ such that if $i\in\Nn$ then $c_{\sig}(\sig^{(i+1)})<\max\{c_{\sig}(\sig^{(i)}),1\}$, that every maximal element of $\sig^{(i)}$ lies in $\bigcup_{\rho\in\sig_1}\sig^{(i)}_{\rho}$, and that every maximal element of $\bigcup_{\rho\in\sig_1\setminus C_{\sig}(\sig^{(i)})}\sig^{(i)}_{\rho}$ is full.
\end{lemma}

\begin{proof}
We construct such a sequence by recursion starting with $\sig^{(0)}\dfgl\sig$. Let $i\in\Nn$ and suppose that there is a sequence $(\sig^{(j)})_{j=0}^i$ with the desired properties. If $c_{\sig}(\sig^{(i)})=0$ then we may set $\sig^{(i+1)}\dfgl\sig^{(i)}$. Otherwise there is a $\xi\in C_{\sig}(\sig^{(i)})$ so that the $W_{\xi}$-fan $\Tau\dfgl\sig/\xi$ is noncomplete. By hypothesis there exists a strong $W_{\xi}$-completion $(\tauquer,\tauhut)$ of $\Tau$, and $\tauhut$ is relatively simplicial and separable over $\Tau$ (\ref{quasilemma}). So, there exists a $W$-pullback datum $(q,a,B)$ along $\xi$ that is very good for $\sig$ and $\tauhut$ (\ref{pullex}), and $\sig^{(i+1)}\dfgl\sig^{(i)}_{q,a,B}(\tauhut)$ is a $W$-fan as desired (\ref{pulla}, \ref{pullb}, \ref{pullc}).
\end{proof}

\begin{lemma}\label{8.20}
Let $\sig$ be a $W$-fan with $\sig_1\neq\emptyset$. If $\sig$ has a $W$-packing then it has a strong $W$-completion.
\end{lemma}

\begin{proof}
Let $\sigquer$ be a $W$-packing of $\sig$, necessarily equifulldimensional. If $\sigquer$ is complete then $(\sigquer,\sigquer)$ is a strong $W$-completion of $\sig$ (\ref{scexas}~b)). Therefore, we suppose that $\sigquer$ is noncomplete. We can choose a $W$-separating family $H$ for $\sig$ (\ref{comp20}) and consider the complete $W$-semifan $\omegaq\dfgl\omegaq_{\sigquer,H}$ associated with $\sigquer$ and $H$. Then, $\Tau\dfgl\{\sigma\in\omegaq\mid\sigma\subseteq\cl(V\setminus|\sigquer|)\}$ is a $W$-fan with $|\sigquer|\cup|\Tau|=V$ and $|\sigquer|\cap|\Tau|=|\fff(\sigquer)|$, and $\Tau'\dfgl\{\sigma\in\omegaq\mid\sigma\subseteq\fr(|\sigquer|)\}$ is a $W$-subfan of $\Tau$ and a $W$-subdivision of $\fff(\sigquer)$ (\ref{comp30}, \ref{top80}). We can choose a simplicial $W$-subdivision $\tauquer$ of $\Tau$, inducing a simplicial $W$-subdivision $\tauquer'$ of $\Tau'$, hence of $\fff(\sigquer)$ (\ref{sub10}). As $\fff(\sigquer)=\exte_{\sig}(\sigquer)$ (\ref{fff}) it follows that $(\adj_{\sig}(\sigquer,\tauquer'),\adj_{\sig}(\sigquer,\tauquer')\cup\tauquer)$ is a strong $W$-completion of $\sig$ (\ref{adjthm2}~d), \ref{adjthm1}~b), \ref{adjthm3}).
\end{proof}

\begin{theorem}\label{8.30}
Every $W$-fan has a strong $W$-completion.
\end{theorem}

\begin{proof}
Let $\sig$ be a $W$-fan. If $\sig_1=\emptyset$ then $\sig$ has a strong $W$-completion by \ref{scexas}~c). We suppose $\sig_1\neq\emptyset$, hence $n\geq 1$, and show the claim by induction on $n$. If $n=1$ it holds by \ref{scexas}~b), d). Let $n>1$ and suppose the claim to hold for strictly smaller values of $n$. Let $(\sig^{(i)})_{i\in\Nn}$ be a sequence as proven to exist in \ref{8.10}. There is an $i\in\Nn$ with $c_{\sig}(\sig^{(i)})=0$, so that $\sig^{(i)}$ is a relatively simplicial, tightly separable $W$-quasipacking of $\sig$ (\ref{qpc})). Since its maximal elements are maximal elements of $\bigcup_{\rho\in\sig_1\setminus C_{\sig}(\sig^{(i)})}\sig^{(i)}_{\rho}$, hence full, it is equifulldimensional and therefore a $W$-packing of $\sig$. Now, the claim follows from \ref{8.20}.
\end{proof}

\begin{corollary}
Every (simplicial) $W$-semifan has a (simplicial) $W$-comple\-tion.
\end{corollary}

\begin{proof}
This is clear by \ref{proj}, \ref{8.30}, \ref{quasilemma} and \ref{pack20}.
\end{proof}


\smallskip

\subsection*{Acknowledgement}

I thank Markus Brodmann, Stefan Fumasoli, and {\fontencoding{T5}\selectfont Ng\ocircumflex{} Vi\d\ecircumflex{}t Trung} for their help during the writing of this article.



\begin{thebibliography}{99}
\bibitem{a} N. BOURBAKI, {\it \'El\'ements de math\'ematique: Alg\`ebre, Chapitres 1 \`a 3.} Hermann, Paris (1970)
\bibitem{ew0} G. EWALD, {\it Spherical complexes and nonprojective toric varieties.} Discrete Comput. Geom. 1 (1986), 115--122 
\bibitem{ew} G. EWALD, {\it Combinatorial convexity and algebraic geometry.} Grad. Texts in Math. 168. Springer-Verlag, New York (1996)
\bibitem{ei} G. EWALD, M.-N. ISHIDA, {\it Completion of real fans and Zariski-Riemann spaces.} Tohoku Math. J. 58 (2006), 189--218
\bibitem{oda} T. ODA, {\it Convex bodies and algebraic geometry.} Translated from the Japanese. Ergebnisse der Mathematik und ihrer Grenzgebiete 15. Springer-Verlag, Berlin (1988)
\bibitem{diss} F. ROHRER, {\it Toric schemes.} Dissertation. Universit\"at Z\"urich, Z\"urich (2010) (available at {\tt www.dissertationen.uzh.ch})
\bibitem{smi} Z. SMILANSKY, {\it Decomposability of polytopes and polyhedra.} Geom. Dedicata 24 (1987), 29--49
\bibitem{sumihiro} H. SUMIHIRO, {\it Equivariant completion.} J. Math. Kyoto Univ. 14 (1974), 1--28
\bibitem{wlo} J. W{\L}ODARCZYK, {\it Toroidal varieties and the weak factorization theorem.} Invent. Math. 154 (2003), 223--331 
\end{thebibliography}
\end{document}